\documentclass[11pt]{article}

\usepackage[utf8]{inputenc} 
\usepackage[T1]{fontenc}    
\usepackage{hyperref}       
\usepackage{url}            
\usepackage{booktabs}       
\usepackage{amsmath, amssymb, amsthm, amsfonts}       
\usepackage{microtype}      
\usepackage{xcolor}         
\usepackage{enumitem}
\usepackage{graphicx}
\usepackage[sort&compress]{natbib}
\usepackage{fullpage}
\usepackage{mathtools, bm, xfrac}
\usepackage{wrapfig}
\usepackage{authblk,textcomp}

\usepackage{libertine}

\hypersetup{pdfauthor={ENS},pdftitle={A unifying approach to SGD},%
            colorlinks, linktocpage=true, pdfstartpage=1, pdfstartview=FitV,%
    breaklinks=true, pdfpagemode=UseNone, pageanchor=true, pdfpagemode=UseOutlines,%
    plainpages=false, bookmarksnumbered, bookmarksopen=true, bookmarksopenlevel=1,%
    hypertexnames=true, pdfhighlight=/O,%
    urlcolor=orange, linkcolor=blue, citecolor=blue
        }

\title{Universality laws for Gaussian mixtures in generalized linear models}
\usepackage{times}
\usepackage{caption}

\author[1]{Yatin Dandi}
\author[1]{Ludovic Stephan}
\author[1]{Florent Krzakala}
\author[2]{Bruno Loureiro}
\author[3]{Lenka Zdeborová}
\affil[1]{\small \'Ecole Polytechnique F\'ed\'erale de Lausanne (EPFL),
  IdePHICS Lab,   CH-1015 Lausanne, Switzerland}
\affil[2]{\small D\'epartement d'Informatique, \'Ecole Normale Sup\'erieure (ENS) - PSL \& CNRS, 
F-75230 Paris cedex 05, France}
\affil[3]{\small \'Ecole Polytechnique F\'ed\'erale de Lausanne (EPFL),
  SPOC Lab,   CH-1015 Lausanne, Switzerland}

\usepackage{mathtools}
\usepackage{etoolbox}

\newcommand{\dR}{\mathbb{R}}

\newcommand{\dE}{\mathbb{E}}
\newcommand{\dS}{\mathbb{S}}
\newcommand{\dP}{\mathbb{P}}
\newcommand{\cA}{\mathcal{A}}
\newcommand{\cB}{\mathcal{B}}
\newcommand{\cC}{\mathcal{C}}

\newcommand{\cK}{\mathcal{K}}
\newcommand{\cL}{\mathcal{L}}

\newcommand{\cN}{\mathcal{N}}

\newcommand{\cR}{\mathcal{R}}
\newcommand{\cS}{\mathcal{S}}

\newcommand{\E}{\mathbb{E}}

\newcommand{\dd}{\mathrm{d}}

\providecommand{\given}{}
\DeclarePairedDelimiterXPP{\Pb}[1]{\mathbb{P}}{\lparen}{\rparen}{}{\renewcommand{\given}{\nonscript{}\:\delimsize\vert\nonscript{}\:\mathopen{}} #1}

\DeclarePairedDelimiterX{\Set}[1]\lbrace\rbrace{\renewcommand{\given}{\nonscript{}\:\delimsize\vert\nonscript{}\:\mathopen{}} #1}

\DeclareMathOperator{\tr}{tr}

\DeclareMathOperator*{\argmin}{arg\,min}

\DeclarePairedDelimiterX{\norm}[1]\lVert\rVert{\ifblank{#1}{\: \cdot \:}{#1}}

\newcommand{\bilingualcommand}[3]{%
	\newcommand{#1}[1][\ ]{%
		##1%
		\iflanguage{english}{\text{#2}}{%
			\iflanguage{french}{\text{#3}}{}%
		}%
		##1%
	}%
}

\bilingualcommand{\where}{where}{où}
\bilingualcommand{\textif}{if}{si}
\bilingualcommand{\textand}{and}{et}
\bilingualcommand{\textiff}{if and only if}{si et seulement si}
\bilingualcommand{\otherwise}{otherwise}{sinon}

\newcommand{\eps}{\varepsilon}

\newcommand{\quand}{\quad \textand{} \quad}
\newcommand{\qquand}{\qquad \textand{} \qquad}

\DeclareMathOperator{\Prox}{Prox}

\newcommand{\bsA}{{\boldsymbol{\mathsf A}}}
\newcommand{\bG}{{\boldsymbol{G}}}

\newcommand{\bg}{{\boldsymbol{g}}}
\newcommand{\bH}{{\boldsymbol{H}}}

\newcommand{\bU}{{\boldsymbol{U}}}

\newcommand{\bh}{{\boldsymbol{h}}}
\newcommand{\bW}{{{\boldsymbol{{W}}}}}
\newcommand{\bw}{{\boldsymbol{w}}}
\newcommand{\bQ}{{\boldsymbol{Q}}}
\newcommand{\bV}{{\boldsymbol{V}}}

\newcommand{\bM}{{\boldsymbol{m}}}
\newcommand{\bhQ}{{\hat{\boldsymbol{Q}}}}
\newcommand{\bhV}{{\hat{\boldsymbol{V}}}}
\newcommand{\bhM}{{\hat{\boldsymbol{m}}}}

\newcommand{\bbb}{{\boldsymbol{b}}}
\newcommand{\bv}{{\boldsymbol{v}}}

\newcommand{\bbB}{{\boldsymbol{B}}}
\newcommand{\bF}{{\boldsymbol{F}}}
\newcommand{\bZ}{{\boldsymbol{Z}}}

\newcommand{\bTheta}{{\boldsymbol{\Theta}}}
\newcommand{\bSigma}{{\boldsymbol{\Sigma}}}

\newcommand{\bomega}{{\boldsymbol{\omega}}}

\newcommand{\bXi}{{\boldsymbol{\Xi}}}
\newcommand{\bmu}{{\boldsymbol{\mu}}}

\newcommand{\bx}{{\boldsymbol{x}}}

\newcommand{\bX}{{\boldsymbol{X}}}
\newcommand{\bI}{{\boldsymbol{I}}}
\newcommand{\by}{{\boldsymbol{y}}}
\newcommand{\bff}{{\boldsymbol{f}}}
\newcommand{\be}{{\boldsymbol{e}}}

\newcommand{\bxi}{{\boldsymbol{\xi}}}

\newcommand{\bu}{{\boldsymbol{u}}}

\DeclareMathOperator{\polylog}{polylog}

\makeatletter
\newtheorem*{rep@theorem}{\rep@title}
\newcommand{\newreptheorem}[2]{%
\newenvironment{rep#1}[1]{%
 \def\rep@title{#2 \ref{##1}}%
 \begin{rep@theorem}}%
 {\end{rep@theorem}}}
\makeatother

\newtheorem{theorem}{Theorem}
\newtheorem{assump}{Assumption}
\newtheorem{lemma}[theorem]{Lemma}
\newtheorem{proposition}[theorem]{Proposition}

\newtheorem{corollary}[theorem]{Corollary}
\newtheorem{definition}[theorem]{Definition}

\providecommand{\bB}{\mathbf{B}}

\providecommand{\bI}{\mathbf{I}}

\providecommand{\xx}{\mathbf{x}}
\providecommand{\bmu}{\boldsymbol \mu}
\providecommand{\bSigma}{\boldsymbol \Sigma}
\providecommand{\btheta}{\boldsymbol{\theta}}
\providecommand{\yy}{\mathbf{y}}
\providecommand{\zz}{\mathbf{z}}
\providecommand{\gg}{\mathbf{g}}

\newcommand{\R}{\mathbb{R}}

\def\hR{\widehat{\mathcal{R}}}

\newcommand{\Ea}[1]{\E\left[#1\right]}
\newcommand{\Eb}[2]{\E_{#1}\left[#2\right]}

\DeclarePairedDelimiterX{\rbr}[1]{(}{)}{#1} 
\DeclarePairedDelimiterX{\sbr}[1]{[}{]}{#1}

\def\bv{{\boldsymbol v}}
\def\bd{{\boldsymbol d}}

\def\bQ{{\boldsymbol Q}}

\def\bZ{{\boldsymbol Z}}

\def\bM{{\boldsymbol M}}

\def\bB{{\boldsymbol B}}

\def\bh{{\boldsymbol h}}
\def\bF{{\boldsymbol F}}
\def\bH{{\boldsymbol H}}
\def\bI{{\boldsymbol I}}
\def\bX{{\boldsymbol X}}

\def\bw{{\boldsymbol w}}
\def\bv{{\boldsymbol v}}

\def\bx{{\boldsymbol x}}
\def\by{{\boldsymbol y}}
\def\bW{{\boldsymbol W}}
\def\bV{{\boldsymbol V}}

\def\bv{{\boldsymbol v}}

\def\bU{{\boldsymbol U}}
\def\bV{{\boldsymbol V}}
\def\bz{{\boldsymbol z}}

\def\bxi{{\boldsymbol \xi}}
\newcommand{\lin}[1]{\ensuremath{\left\langle #1 \right\rangle}}
\newcommand{\abs}[1]{\left\lvert#1\right\rvert}

\def\mat#1{\mathrm{#1}}

\begin{document}

\maketitle

\begin{abstract}%
Let $(\bx_{i}, y_{i})_{i=1,\dots,n}$ denote independent samples from a general mixture distribution $\sum_{c\in\mathcal{C}}\rho_{c}P_{c}^{\bx}$, and consider the hypothesis class of generalized linear models $\hat{y} = F(\bTheta^{\top}\bx)$. In this work, we investigate the asymptotic joint statistics of the family of generalized linear estimators $(\bTheta_{1}, \dots, \bTheta_{M})$ obtained either from (a) minimizing an empirical risk $\hat{R}_{n}(\bTheta;\bX,\by)$ or (b) sampling from the associated Gibbs measure $\exp(-\beta n \hat{R}_{n}(\bTheta;\bX,\by))$. Our main contribution is to characterize under which conditions the asymptotic joint statistics of this family depends (on a weak sense) only on the means and covariances of the class conditional features distribution $P_{c}^{\bx}$. In particular, this allow us to prove the universality of different quantities of interest, such as the training and generalization errors, redeeming a recent line of work in high-dimensional statistics working under the Gaussian mixture hypothesis. Finally, we discuss the applications of our results to different machine learning tasks of interest, such as ensembling and uncertainty quantification. 
\end{abstract}

\section{Introduction}
A recurrent topic in high-dimensional statistics is the investigation of the typical properties of signal processing and machine learning methods on synthetic, {\it i.i.d.} Gaussian  data, a scenario often known under the umbrella of \emph{Gaussian design} \citep{donoho2009observed, 5730603, Monajemi2013, candes2020phase, Bartlett30063}.  A less restrictive assumption arises when considering that many  machine learning tasks deal with data partitioned into a fixed number of classes. In these cases, the data distribution is naturally described by a \emph{mixture model}, where each sample is generated \emph{conditionally} on the class. In other words: data is generated by first sampling the class assignment and \emph{then} generating the input conditioned on the class. Arguably the simplest example of such distributions is that of a {\it Gaussian mixture}, which shall be our focus in this work. 

Gaussian mixtures are a popular model in high-dimensional statistics since, besides being an universal approximator, they often lead to mathematically tractable problems. Indeed, a recent line of work has analyzed the asymptotic performance of a large class of machine learning problems in the proportional high-dimensional limit under the Gaussian mixture data assumption, see e.g. 
\cite{mai2019high,mignacco2020role,taheri2020optimality,kini2021phase,wang2021benign,refinetti2021classifying,loureiro2021learning_gm}. The key goal of the present work is to show that this assumption, and hence the conclusions derived therein, are far more general than previously anticipated.

We build on a recent line of works  \citep{goldt2020gaussian,hu2020universality,montanari2022universality} that have proven the asymptotic (single) Gaussian equivalence of generalized linear estimation on non-linear feature maps satisfying certain regularities conditions, a topic that has started with the work of \cite{el2010spectrum} on kernel matrices. Furthermore, there is strong empirical evidence that Gaussian universality holds in a more general sense \citep{loureiro2021learning}. A crucial limitation of the results in \cite{hu2020universality,montanari2022universality}, however, is the assumption of a target function depending on linear projections in the latent or feature space. Instead, we consider a rich class of mixture distributions, allowing arbitrary dependence between the class labels and the data.

Here, we extend this line of works and provide rigorous justification for universality in various settings such as empirical risk minimization (ERM), sampling, ensembling, etc. for {\it general mixture} distributions. Namely, we shall show that the statistics of generalized estimators obtained either from ERM or sampling on a mixture model asymptotically agrees (in a weak sense) with the statistics of estimators from the same class trained on a Gaussian mixture model with matching first and second order moments. In particular, this implies the universality of different quantities of interest, such as the training and generalization errors.




Our {\bf main contributions} are as follows:
\begin{itemize}[wide=1pt, topsep=0pt,nosep]
    \item We extend the Gaussian universality of empirical risk minimization theorems in \cite{goldt2020gaussian,hu2020universality,montanari2022universality} to generic mixture distribution and an equivalent mixture of Gaussians. In particular, we show that a Gaussian mixture observed through a random feature map is also a Gaussian mixture in the high-dimensional limit, a fact used for instance (without rigorous justification) in \cite{refinetti2021classifying,loureiro2021learning_gm}.
    \item A consequence of our results is that, with conditions on the matrix weights, data generated by conditional Generative Adversarial Networks (cGAN) behave as a Gaussian mixture when observed through the prism of generalized linear models (kernels, feature maps, etc...), as illustrated in Figs \ref{fig:cartoon_theorem1} and \ref{fig:experiments}. This further generalizes the work of \cite{seddik_2020_random} that only considered the universality of Gram matrices for GAN generated data through the prism of random matrix theory. 
    \item We consider setups involving multiple sets of parameters arising from simultaneous minimization of different objectives as well as sampling from Gibbs distributions defined by the empirical risk. This provides a unified framework for establishing the asymptotic universality of arbitrary functions of the set of minimizers or samples from different Gibbs distributions. For instance, it includes ensembling \citep{loureiro2022fluc}) and uncertainty quantification \citep{Clarte2022a,Clarte2022b} settings.    
    \item We finally show how, in common setups, universality holds for a large class of functions, leading to the equivalence between the distributions of the minimizers themselves, and provide a theorem for their weak convergence.
\end{itemize}

\paragraph{Related work ---} Universality is an important topic in applied mathematics, as it motivates the scope of tractable mathematical models. It has been extensively studied in the context of random matrix theory \citep{Tao2011, Tao2012}, signal processing problems \citep{donoho2009observed, 5730603, Monajemi2013, NIPS2017_136f9513, 8006947, NEURIPS2019_dffbb6ef, Abbara_2020, Dudeja2022} and kernel methods \citep{el2010spectrum, Lu2022, Misiakiewicz2022}. Closer to us is the recent stream of works that investigated the Gaussian universality of the asymptotic error of generalized linear models trained on non-linear features, starting from single-layer random feature maps \citep{montanari2019generalization, gerace2020generalisation, hu2020universality, dhifallah2021} and later extended to single-layer NTK \citep{montanari2022universality} and deep random features \citep{Schroder2023}. These results, together with numerical observations that Gaussian universality holds for more general classes of features, led to the formulation of different \emph{Gaussian equivalence} conjectures \citep{goldt2019modelling, goldt2020gaussian, loureiro2021learning}. A complementary line of research has investigated cases in which the distribution of the features is multi-modal, suggesting a Gaussian mixture universality class instead \citep{louart2018random, seddik_2020_random, seddik_2021_unexpected}. A bridge between these two lines of work has been recently investigated for generalized linear estimation with random labels in \cite{gerace2022}. 

\section{Setting and motivation}
\label{sec:set}
Consider a supervised learning problem where the training data $(\bm x_{i}, y_{i})\in\mathbb{R}^{p}\times \mathcal{Y}$, $i\in [n]\coloneqq \{1,\cdots, n\}$ is independently drawn from a mixture distribution:
\begin{align}
\label{def:mixture}
    {\bm x}_{i} &\sim \sum_{c\in\mathcal{C}} \rho_{c}P^{\bm x}_{c}, &&\mathbb{P}(c_{i}=c) = \rho_c,
\end{align}
with $c_{i}$ a categorical random variable denoting the cluster assignment for the $i_{th}$ example $\xx_i$. Let $\bmu_c$, $\bSigma_c$ denote the mean and covariance of $P_{c}^{\bm{x}}$, and $k=|\mathcal{C}|$.
Further, assume that the labels $y_{i}$ are generated from the following target function:
\begin{align}
\label{eq:def:teacher}
y_{i}(\bX) = \eta(\bm \Theta_{\star}^{\top} \bm x_i, \eps_i, c_i),
\end{align}
where $\bm\Theta_{\star}\in \mathbb{R}^{k\times p}$ and $\eps_i$ is an i.i.d source of randomness. It is important to stress that the class labels \eqref{eq:def:teacher} are themselves not constrained to arise from a simple function of the inputs $\bx_{i}$. For instance, the functional form in \eqref{eq:def:teacher} includes the case where the labels are exclusively given by a function of the mixture index $y_{i}=\eta(c_{i})$. This will allow us to handle complex targets, such as data generated using conditional Generative Adversarial Networks (cGANs).

In this manuscript, we will be interested in the hypothesis class defined by the following parametric predictor $\hat{y} = F(\bTheta^{\top}\bx)$, where $\bTheta\in\mathbb{R}^{k\times p}$ are the parameters and  $F:\mathbb{R}^{k}\to \mathcal{Y}$ an activation function. For a given loss function $\ell:\mathbb{R}^{k}\times \mathcal{Y}\to \mathbb{R}_{+}$ and regularization term $r:\mathbb{R}^{k\times p}\to\mathbb{R}_{+}$, define the (regularized) empirical risk over the training data:
\begin{equation}
\label{eq:def:emprisk}
    \hR_n(\bTheta;\bX,\by) := \frac{1}{n}\sum_{i=1}^n
\ell(\bTheta^{\top}\bx_{i},y_{i})+r(\bTheta),
\end{equation}
\noindent where we have defined the feature matrix $\bX\in\mathbb{R}^{p\times n}$ by stacking the features $\bx_{i}$ column-wise and the labels $y_{i}$ in a vector $\by\in\mathcal{Y}^{n}$. In what follows, we will be interested in the following two tasks:
\begin{enumerate}[wide=1pt, topsep=0pt,nosep]
    \item Minimization: in a minimization task, the statistician's goal is to find a good predictor by minimizing the empirical risk \eqref{eq:def:emprisk}, possibly over a constraint set $\mathcal{S}_{p}$:
\begin{equation}
\label{def:min_problem_committee}
    \hat\bTheta_{\text{erm}}(\bX, \by) \in \underset{\bTheta\in\mathcal{S}_{p}}{\argmin}~\hR_n(\bTheta;\bX,\by),
\end{equation}
This encompasses diverse settings such as generalized linear models with noise, mixture classification, but also the random label setting (with $\eta(x, \eps) = \eps$). In the following, we denote {$\hR_{n}^{\star}(\bm X, \bm y)\coloneqq \min_{\bTheta} \hR_{n}(\bTheta;\bm X, \bm y)$}

    \item Sampling: here, instead of minimizing the empirical risk \eqref{eq:def:emprisk}, the statistician's goal is to sample from a Gibbs distribution that weights different hypothesis according to their empirical error:
    \begin{equation}
    \label{eq:def:gibbs}
        \bTheta_{\text{Bayes}}(\bX,\by) \sim P_{\text{Bayes}}(\bTheta) \propto \exp\left(-\beta n\hR_n(\bTheta;\bX,\by)\right) \dd\mu(\bTheta)
    \end{equation}
    where $\mu$ is reference prior measure and $\beta>0$ is a parameter known as the \emph{inverse temperature}. Note that minimization can be seen as a particular example of sampling when $\beta\to\infty$, since in this limit the above measure peaks on the global minima of \eqref{def:min_problem_committee}.  
\end{enumerate}
\paragraph{Applications of interest---}
So far, the setting defined above is quite generic, and the motivation to study this problem might not appear evident to the reader. Therefore, we briefly discuss a few scenarios of interest which are covered this model.
\begin{enumerate}[wide=1pt, topsep=0pt,nosep]
    \item \textit{Conditional GANs (cGANs):} were introduced by \cite{Mirza2014} as a generative model to learn mixture distributions. Once trained in samples from the target distribution, they define a function $\Psi$ that maps Gaussian mixtures (defining the latent space) to samples from the target mixture that preserve the label structure. In other words, conditioned on the label:
    \begin{align}
        \forall c\in\mathcal{C}, \qquad \bz\sim\mathcal{N}(\bmu_{c}, \bSigma_{c})\mapsto \bx_{c}=\Psi(\bz, c)\sim P_{c}^{\bx}
    \end{align}
    The connection to model \eqref{def:mixture} is immediate. This scenario was extensively studied by \cite{louart2018random, seddik_2020_random, seddik_2021_unexpected}, and is illustrated in Fig.~\ref{fig:cartoon_theorem1}. In Fig.~\ref{fig:experiments} we report on a concrete experiment with a cGAN trained on the fashion-MNIST dataset. 

    \item\textit{Multiple objectives:} Our framework also allows to characterize the joint statistics of estimators $(\bTheta_{1},\dots, \bTheta_{M})$ obtained from empirical risk minimization and/or sampling from different objective functions $\hat{R}^{m}_{n}$ defined on the same training data $(\bX, \by)$. This can be of interest in different scenarios. For instance, \cite{Clarte2022a, Clarte2022b} has characterized the correlation in the calibration of different uncertainty measures of interest, e.g. last-layer scores and Bayesian training of last-layer weights. This crucially depends on the correlation matrix $\hat{\bTheta}_{\text{erm}}\bTheta_{\text{Bayes}}^{\top}\in\mathbb{R}^{k\times k}$ which fits our framework.
    
    \item \textit{Ensemble of features:} Another example covered by the multi-objective framework above is that of ensembling. Let $(\bz_{i}, y_{i})\in\mathbb{R}^{d}\times \mathcal{Y}$ denote some training data from a mixture model akin to \eqref{def:mixture}. A popular  ensembling scheme often employed in the context of deep learning \citep{NIPS2017_9ef2ed4b} is to take a family of $M$  feature maps $\bz_{i} \mapsto \bx^{(m)}_{i} \!=\!\varphi_{m}(\bz_{i})$ (e.g.  neural network features trained from different random initialization) and train $M$ independent learners:
    \begin{align}
        \hat{\bm{\Theta}}_{\text{erm}}^{(m)} \in \underset{\bm{\bTheta}\in\mathcal{S}_{p}}{\argmin}~\frac{1}{n}\sum\limits_{i=1}^{n}\ell(\bm{\Theta}^{\top}\bx^{(m)}_{i}, y_{i}) + r(\bm{\bTheta})
    \end{align}
    Prediction on a new sample $\bz$ is then made by assembling the independent learners, e.g. by taking their average $\hat{\by} = \sfrac{1}{M}\sum_{m=1}^{M}{\hat{\bm{\Theta}}_{\text{erm}}^{(m)\top}}\varphi_{m}(\bz)$. A closely related model was studied in \cite{Geiger_2020, d2020double, loureiro2022fluc}.
\end{enumerate}

Note that in all the applications above, having the labels depending on the features $\bX$ would not be natural, since they are either generated from a latent space, as in $(i)$, or chosen by the statistician, as in $(ii), (iii)$. Indeed, in these cases the most natural label model is given by the mixture index $y=c$ itself, which is a particular case of \eqref{eq:def:teacher}. This highlights the flexibility of the our target model with respect to prior work \citep{montanari2022universality}. Instead, \cite{hu2020universality} assumes that the target is a function of a \emph{latent variable}, which would correspond to a mismatched setting. The discussion here could be generalized also to this case, but would require an additional assumption, which we discuss in Appendix \ref{sec:app:target}.  

\paragraph{Universality ---} Given these tasks, the goal of the statistician is to characterize different statistical properties of these predictors. These can be, for instance, point performance metrics such as the empirical and population risks, or uncertainty metrics such as the calibration of the predictor or moments of the posterior distribution \eqref{eq:def:gibbs}. These examples, as well as many different other quantities of interest, are functions of the joint statistics of the pre-activations $(\bTheta_{\star}^{\top}\bx, \bTheta^{\top}\bx)$, for $\bx$ either a test or training sample from \eqref{def:mixture}. For instance, in a Gaussian mixture model, where $\bx\sim \sum_{c\in\mathcal{C}}\rho_{c}\, \mathcal{N}(\bmu_{c},\bSigma_{c})$, the sufficient statistics are simply given by the first two moments of these pre-activations. However, for a general mixture model \eqref{def:mixture}, the sufficient statistics will generically depend on all moments of these pre-activations. Surprisingly, our key result in this work is to show that in the high-dimensional limit this is not the case. In other words, under some conditions which are made precise in Section \ref{sec:mainres}, we show that expectations with respect to \eqref{def:mixture} can be exchanged by expectations over a Gaussian mixture with matching moments. This can be formalized as follows. Define an \emph{equivalent Gaussian data set} $(\bg_{i}, y_{i})_{i=1}^{n}\in\mathbb{R}^{p}\times \mathcal{Y}$ with samples independently drawn from the \emph{equivalent Gaussian mixture model}:
\begin{align}
\label{eq:gaussian_equivalent_model}
\bm g_i \sim \sum_{c\in\mathcal{C}} \rho_{c}\, \cN(\bm \mu_c, \bm \Sigma_c), && y_{i}(\bG) = \eta(\bm \Theta_{\star}^{\top} \bm g_i, \eps_i, c_i).
\end{align}
We recall that $\bmu_c$, $\bSigma_c$ denotes the mean and covariance of $P_{c}^{\bm{x}}$ from \eqref{def:mixture}. Consider a family of estimators $\left(\bTheta_{1},\cdots,\bTheta_{M}\right)$ defined by minimization \eqref{eq:def:emprisk} and/or sampling \eqref{eq:def:gibbs} over the training data $(\bX, \by)$ from the mixture model \eqref{def:mixture}. Let $h$ be a statistical metric of interest. Then, in the proportional high-dimensional limit where $n,p\!\to\!\infty$ at fixed $k,k, M\in\mathbb{Z}_{+}$  sample complexity $\alpha\!=\!\sfrac{n}{d}>0$, and
where $\langle\cdot\rangle$ denote the expectation with respect to the Gibbs distribution \eqref{eq:def:gibbs}, we define universality as:
\begin{align}
    \mathbb{E}_{\bX}\left[\langle h(\bTheta_{1},\cdots,\bTheta_{M}) \rangle_{\bX}\right] \underset{n\to\infty}{\simeq} \mathbb{E}_{\bG}\left[\langle h(\bTheta_{1},\cdots,\bTheta_{M}) \rangle_{\bG}\right]
\end{align}
\noindent 
The goal of the next section is to make this statement precise.

\begin{figure}[t!]
\begin{center}
\centerline{\includegraphics[width=400pt]{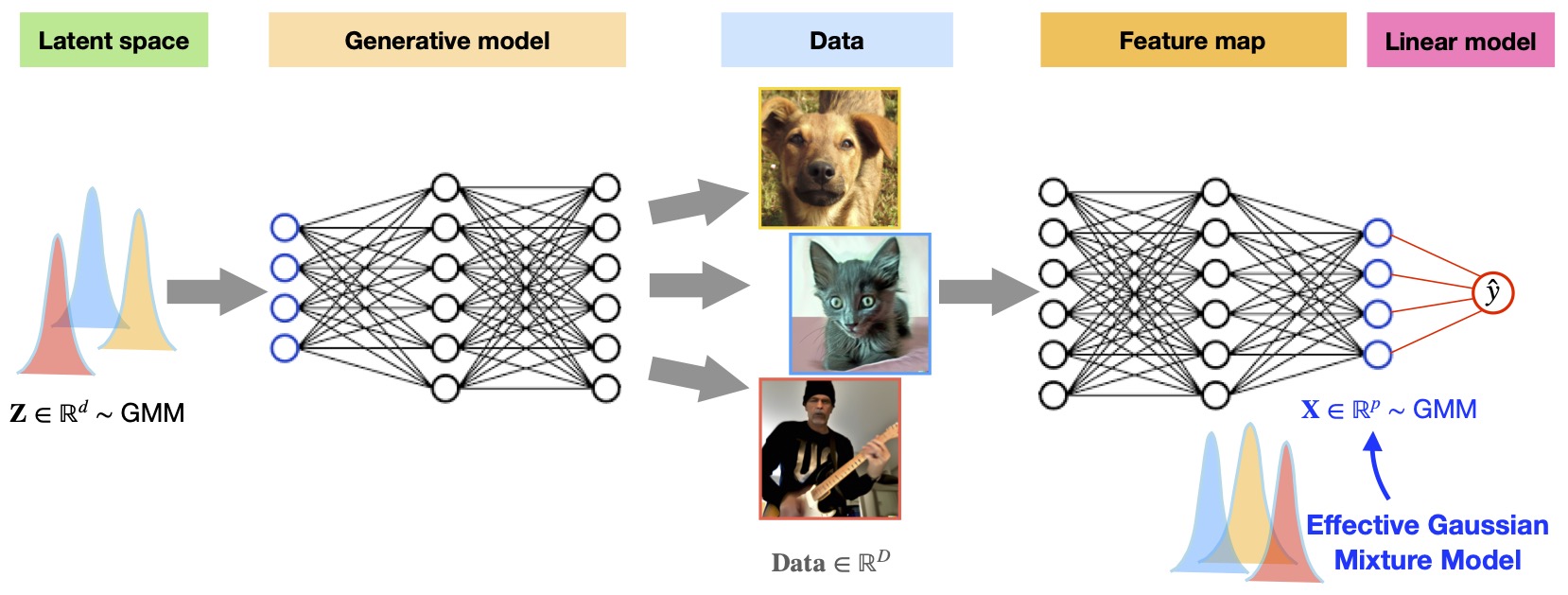}}
\caption{Illustration of Corollary \ref{thm:main_gaussian_universality}: high-dimensional data generated by  generative neural networks starting from a mixture of Gaussian in latent space (${\bf z} \in {\mathbb R}^H$) are (with conditions on the weights matrices) equivalent, in high-dimension and for generalized linear models, to data sampled from a Gaussian mixture. A  concrete example is shown in Fig.
    \ref{fig:experiments}.}
\label{fig:cartoon_theorem1}
\end{center}
\vspace{-1cm}
\end{figure}
\begin{figure}[t]
    \centering
    \includegraphics[width=0.45\textwidth]{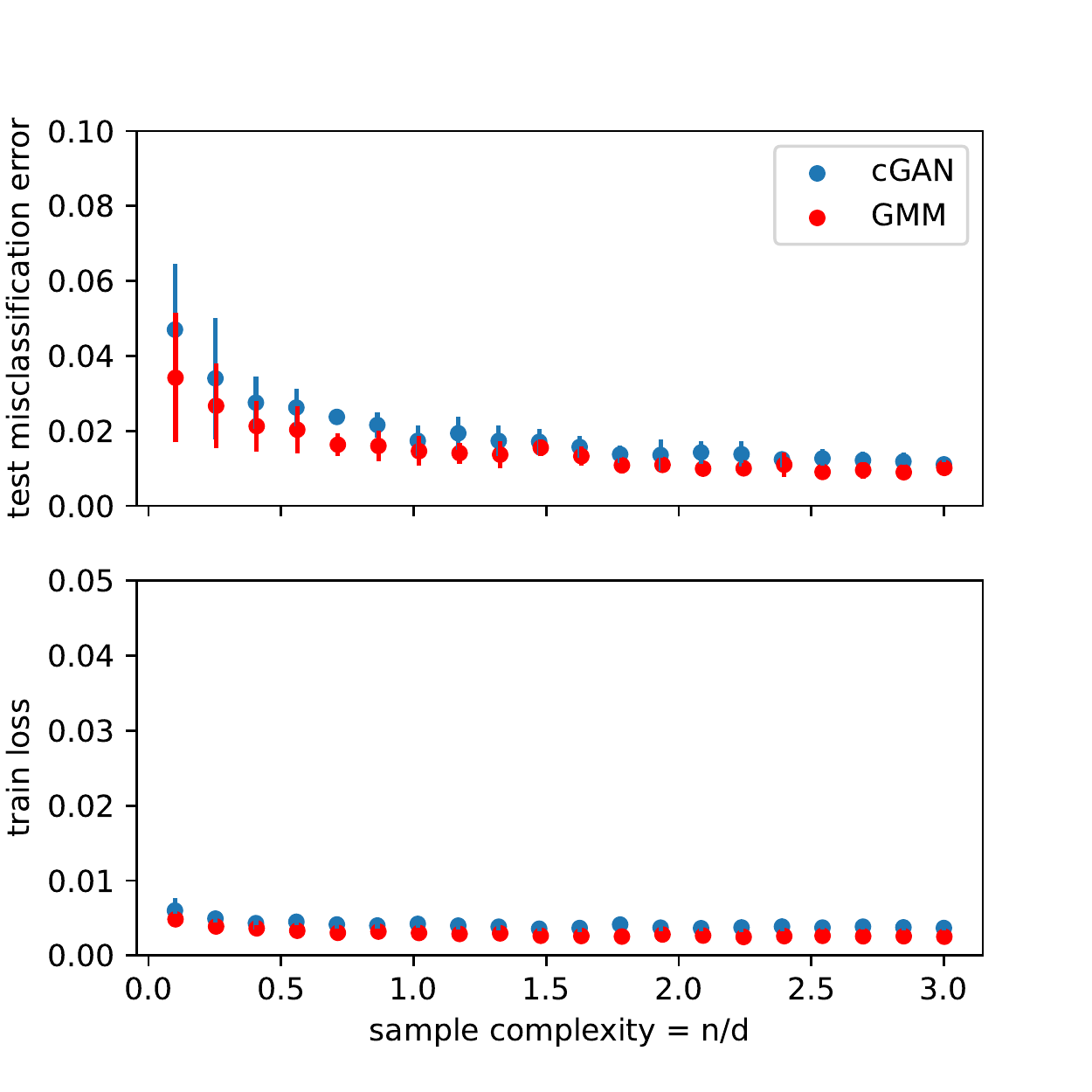}
    \includegraphics[width=0.45\textwidth]{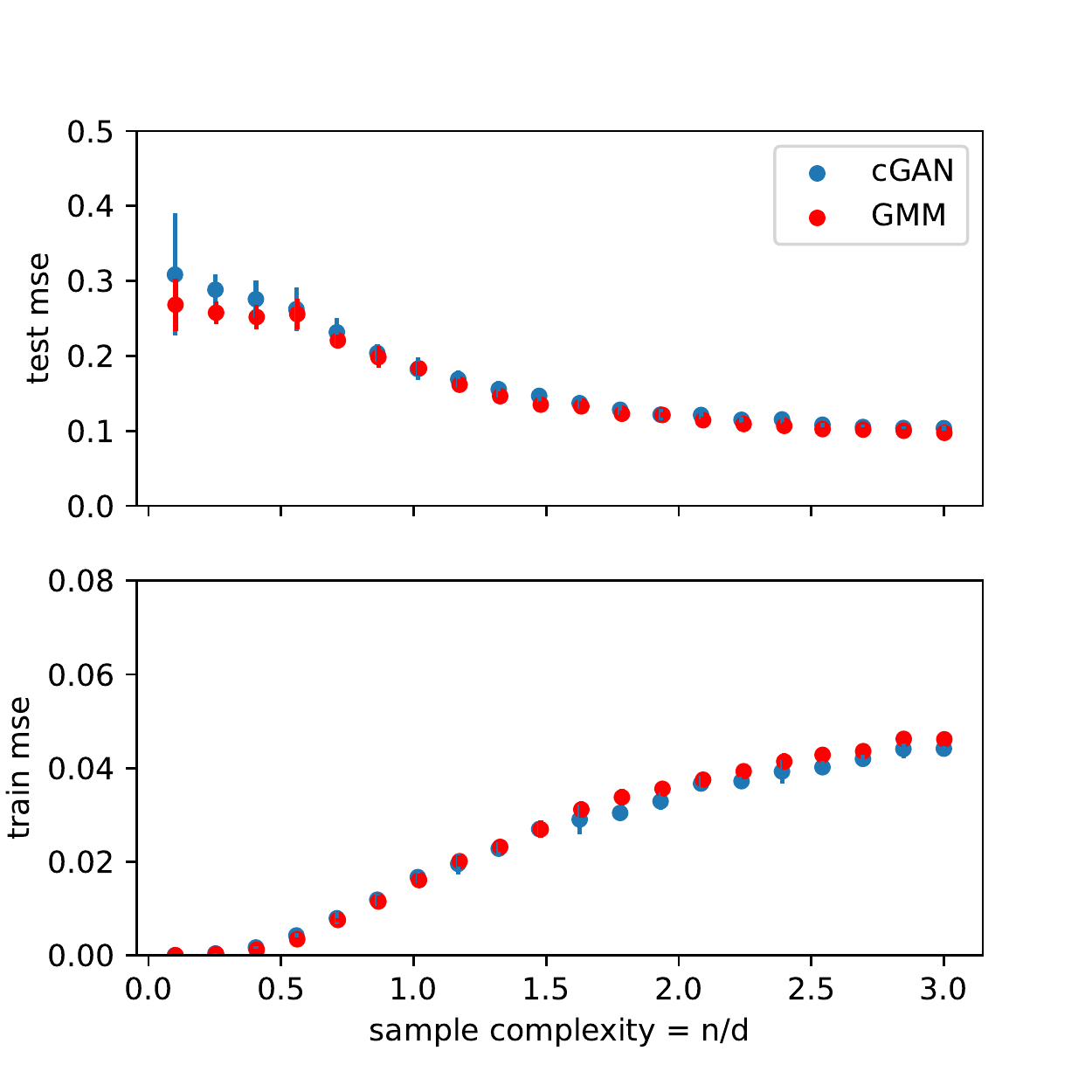}
    \caption{Illustration of the universality scenario described in Fig.\ref{fig:cartoon_theorem1}.  Logistic (left) \& ridge (right) regression test (up) and training (bottom) errors are shown versus the sample complexity $\alpha=\sfrac{n}{d}$ for an odd vs. even binary classification task on two data models: Blue dots data generated from a conditional GAN \citep{Mirza2014} trained on the fashion-MNIST dataset \citep{xiao2017} and pre-processed with a random features map $\bm{x}\mapsto\tanh(\mat{W}\bm{x})$ with Gaussian weights $W\!\in\!\mathbb{R}^{1176\times 784}$; Red dots are the $10$- clusters Gaussian mixture model with means and covariances matching each fashion-MNIST cluster conditioned on labels ($\ell_2$ regularization is $\lambda\!=\!10^{-4}$). Details on the simulations are discussed in Appendix \ref{sec:app:numerics}.}
    \label{fig:experiments}
\end{figure}

\section{Main results} 
\label{sec:mainres}
We now present the main theoretical contributions of the present work and discuss its consequences. We work under the following regularity and concentration assumptions:

\begin{assump}[Loss and regularization]
\label{assump:loss}
    The loss function $\ell : \dR^{k+1} \to \dR$ is nonnegative and Lipschitz, and the regularization function $r : \dR^{p\times k} \to \dR$ is locally Lipschitz, with constants independent from $p$.
\end{assump}

\begin{assump}[Boundedness and concentration]\label{assump:concentration}
    The constraint set $\cS_p$ is a compact subset of $\dR^{k \times p}$. Further, there exists a constant $M > 0$ such that for any $c \geq 0$,
    \begin{equation}
    \sup_{\bm \theta \in \cK_p, \norm{\bm \theta}_2 \leq 1} \norm{\bm \theta^\top \bm x}_{\psi_2} \leq M, \quad \sup_{\bm \theta \in \cK_p, \norm{\bm \theta}_2 \leq 1} \norm{\bSigma_c^{1/2} \bm \theta}_{2} \leq M, \quand \norm{\bmu_c}_2 \leq M
    \end{equation}
    where $\norm{\cdot}_{\psi_2}$ is the sub-gaussian norm, and $\cK_p \subseteq \dR^p$ is such that $\cS_p \subseteq \cK_p^k$.
\end{assump}

\begin{assump}[Labels]\label{assump:labels}
    The labeling function $\eta$ is Lipschitz, the teacher vector $\bTheta$ belongs to $\cS_p$, and the noise variables $\varepsilon_i$ are i.i.d subgaussian with
    $ \norm{\varepsilon_i}_{\psi_2} \leq M $
    for some constant $M > 0$.
\end{assump}

Those three assumptions are fairly technical, and it is possible that the universality properties proven in this article hold irrespective of these conditions. The crucial assumption in our theorems is that of a \emph{conditional one-dimensional CLT}:
\begin{assump}\label{assump:clt}
For any Lipschitz function $\varphi: \dR \to \dR$, 
\begin{equation}\label{eq:one_dimensional_clt_main}
    \lim_{n, p \to \infty} \sup_{\btheta\in\cK_p} 
    \left| 
    \Ea{\varphi(\btheta^\top \bm x)\,\big|\,c_{\bm x}=c} - 
   \Ea{\varphi(\btheta^\top\bm g)\,\big|\,c_{\bm g}=c}
   \right| = 0,\quad \forall c \in \cC
\end{equation}
\end{assump}
where $\bm x$ and $\bm g$ denote samples from the given mixture distribution and the equivalent gaussian mixture distribution in equations \eqref{def:mixture} and \eqref{eq:gaussian_equivalent_model} respectively.

\subsection{Universality of Mixture Models}
We start by proving the universality of the free energy for a Gibbs distribution defined through the objective $\hR_n(\bTheta;\bX,\by)$ for the data distribution defined in (\ref{sec:set}) and its equivalent Gaussian mixture distribution \eqref{eq:gaussian_equivalent_model}. 

\begin{theorem}[Universality of Free Energy]\label{thm:single_univ_free}
Let $\mu_p(\bTheta)$ be a sequence of Borel probability measures with compact supports $\cS_p$. Define the following free energy function:
\begin{equation}
     f_{\beta,n}(\bX)= \int \exp\left(-\beta n \hR_n(\bTheta;\bX,\by(\bX))\right)d\mu_p(\bTheta)
\end{equation}
Under Assumptions \ref{assump:loss}-\ref{assump:clt} on $\bX$ and $\cS_p$, for any bounded differentiable function $\Phi$ with bounded Lipschitz derivative, we have:
\begin{equation}
\nonumber
 \lim_{n, p\to\infty} \left|\Ea{\Phi\left(f_{\beta,n}(\bX)\right)} - 
\Ea{\Phi\left(f_{\beta,n}(\bG)\right)}\right| = 0.
\end{equation}
\end{theorem}

When $\mu_p$ corresponds to discrete measures supported on an $\epsilon$-net in $\cS_p$, using the reduction from Lemma 1 to Theorem 1 in \cite{montanari2022universality}, we obtain the following corollary:
\begin{corollary}\label{corr:gmm_error}
[Universality of Training Error GMM] \label{thm:main_gaussian_universality}
For any bounded Lipschitz function $\Phi: \dR \to \dR$:
\[\lim_{n, p \to \infty} \left| \Ea{\Phi\left(\hR_n^\star(\bm X, \bm y(\bm X))\right)} - \Ea{\Phi\left(\hR_n^\star(\bm G, \bm y(\bm G))\right)} \right| = 0\]
In particular, for any $\mathcal{E} \in \dR$, and denoting 
 $\overset{\dP}{\longrightarrow}$  the convergence in probability:
\begin{equation}\label{eq:in_proba_convergence}
\hR_n^\star(\bm X, \bm y(\bm X)) \overset{\dP}{\longrightarrow} \mathcal{E} \quad \text{if and only if} \quad \hR_n^\star(\bm G, \bm y(\bm G)) \overset{\dP}{\longrightarrow} \mathcal{E},
\end{equation}
\end{corollary}
The full theorem, as well as its proof, is presented in Appendix \ref{app:gaussian_universality}. In a nutshell, this theorem shows that the multi-modal data generated by any generative neural network is equivalent to a {\it finite} mixture of Gaussian in high-dimensions: in other words, a {\it finite} mixture of Gaussians leads to the same loss as  for data generated by (for instance) a cGAN. Since the function $\ell: \mathbb R^{k+1} \to \mathbb R$ need not be convex, we can take
\[ \ell(\bm x_{\text{out}}, y) = \ell'(\bm \Psi(\bm x_{\text{out}}), y), \]
where $\bm \Psi$ is an already pretrained neural network. In particular, if $\bm x$ is the output of all but the last layer of a neural net,  we can view $\bm \Psi$ as the averaging procedure for a small committee machine. 

Note that Corollary \ref{thm:main_gaussian_universality} depends crucially on assumption \ref{assump:clt} (the one-dimensional CLT), which is by no means evident. We discuss the conditions on the weights matrix for which it can be proven in section \ref{sec:1d-CLT}. However, one can observe empirically that the validity of Corollary \ref{thm:main_gaussian_universality} goes well beyond what can be currently proven. A number of numerical illustrations of this property can be found in the work of \cite{seddik_2020_random,louart2018random,seddik_2021_unexpected}, who already derived similar (albeit more limited) results using random matrix theory. Additionally, we observed that even with trained GANs, when we observed data through a random feature map \citep{rahimi2007random}, the Gaussian mixture universality is well obeyed. This scenario is illustrated in Fig.\ref{fig:cartoon_theorem1}, with a concrete example in Fig.\ref{fig:experiments}. Even though we did not prove the one-dimensional CLT for arbitrary learned matrices, and worked with finite moderate sizes, the realistic data generated by our cGAN behaves extremely closely with those of generated by the corresponding Gaussian mixture.

A second remark is that the interest of the Corollary lies in the fact that it requires only a {\it finite} mixture to approximate the loss. Indeed, while  we could use the standard approximation results (e.g. the Stone-Weierstrass theorem) to approximate the data density to arbitrary precision by Gaussian mixtures, this would require a diverging number of Gaussian in the mixture. The fact that loss is captured with finite $\cC$ is key to our approach.

\paragraph*{Proof sketch and remarks ---}

While we provide a complete proof of these results in App. \ref{app:gaussian_universality}, we believe it is useful to present a short intuitive presentation explaining why these results hold. 
\begin{enumerate}[wide=1pt, topsep=1pt,itemsep=0pt]
    \item A crucial first remark is that Theorem \ref{thm:single_univ_free} and Corollary \ref{thm:main_gaussian_universality} do not require that the data generated by GANs are Gaussians mixtures: rather, it is their one-dimensional projections along the directions in ${\cal S}_p$ that should behave as such. The first, intuitive, explanation, is that indeed in high dimension, for a randomly chosen vector ${\bm\theta\in\cS_p}$, it is natural to expect that $\bm\theta^\top \bm x)$ behaves like a Gaussian mixture. Indeed, if we condition on a given label, then $z$ is Gaussian, the  random variable $\bm x = \Psi_{\text{nn}}({\bf z}),$ has a well defined mean and variance (at least if $\Psi$ is a Lipschitz function), so that the central limit theorem shows that $\bm\theta^\top \bm x)$ converges to a Gaussian variable. 
    \item We require, however, a slightly stronger condition in eq.~\eqref{eq:one_dimensional_clt_main}: Indeed, it should be that, conditioned on a label, 
    ${\bm\theta\in\cS_p}$ is Gaussian {\it for all } ${\bm\theta\in\cS_p}$, not a randomly chosen one (since we do certainly do not chose our weights randomly). This condition  might appear strong. However, such one-dimensional CLTs have been the subject of many recent works which proved them for many cases \citep{goldt2020gaussian,hu2020universality,montanari2022universality}, including random features and two-layers neural tangent kernels. 
    We extend the proof of the one-dimensional CLT to mixture models in section \ref{sec:1d-CLT}. We also provide further formal arguments in App. \ref{app:1d_clt_arguments}. In particular, we argue that a large class of distributions, including deep generative models, do satisfy this condition that 
    can also be checked empirically in simulations \citep{goldt2020gaussian,seddik_2021_unexpected}.    
    \item    The one-dimensional CLT  now implies that 
    $\mathbb{E}\left[ \hR_n \left( \bTheta;\bX, \by(\bX) \right) \right] \simeq \mathbb{E}\left[ \hR_n \left( \bTheta; \bG, \by(\bG) \right) \right]$
    for any \textbf{fixed} choice of $\bTheta$, \textbf{independent from the data}. 
    There is another, additional difficulty:  when one performs empirical risk minimization, the minimizer $\hat{\bTheta}(\bX)$ {\bf strongly depends} on the data $\bX$, and the naive 1d-CLT simply does not apply!  Solving the problem of the dependence of the estimator over the data is the main mathematical difficulty in proving Thm. \ref{thm:main_gaussian_universality}. This is achieved in the Appendix by using a method due to \cite{montanari2022universality}, that leverage on the Guerra interpolation techniques used to prove the validity of the replica method \cite{guerra2003broken}.  The idea is to define a t-dependent model that uses $n$ dataset points at $t=0$, and GMM ones at $t=1$, and to show that the free energy (and all observables) remains constants at all ``times'' $t \in [0,1]$). This establishes fully the universality advocated in our theorem.     \hfill \qedsymbol  
\end{enumerate}

\subsection{Convergence of expectations for Joint Minimization and Sampling}

Our next result establishes a general relationship between the differentiability of the limit of expected training errors or free energies for empirical risk minimization or free energies for sampling and the universality of expectations of a given function of the parameters.

\paragraph{Setup} Consider a sequence of $M$ risks
\begin{equation}
     \hR^{(m)}_n(\bTheta;\bX^{(m)},\by^{(m)}) := \frac{1}{n}\sum_{i=1}^n
\ell_m(\bTheta^{\top}\bx_i^{(m)},y^{(m)}_i)+r_m(\bTheta), \quad m\in [M]
\end{equation}
with possibly different losses, regularizers and datasets. For simplicity, we assume that the objectives are defined on parameters having the same dimension $\bTheta\in\mathbb{R}^{p\times k}$. We aim to minimize $M_1$ of them:
\begin{equation}\label{def:min_multi}
    \hat{\bTheta}^{(m)}(\bX) \in \argmin_{\bTheta \in \cS_p^{(m)}}\  \hR^{(m)}_n(\bTheta;\bX^{(m)},\by^{(m)}), \quad m \in [M_1]
\end{equation}
and the $M_2$ remaining parameters are independently sampled from a family of Gibbs distributions:
\begin{equation}\label{def:sampling}
    \bTheta^{(m)}\sim P_m(\bm \Theta) \propto \exp{\rbr*{-\beta_m \hR^{(m)}_n\left(\bTheta;\bX^{(m)},\by^{(m)}\right)}d\mu_m(\bTheta)},  \quad m \in [M_1+1,M],
\end{equation}
where $M=M_1+M_2$.
The joint distribution of the $\bx_i = (\bx_i^{(1)}, \dots, \bx_i^{(M)})$ is assumed to be a mixture of the form \eqref{def:mixture}. However, we assume that the labels $\by_i^{(m)}$ only depend on the vectors $\bx_i^{(m)}$:
\begin{equation}
y^{(m)}_{i}(\bX^{(m)}) = \eta(\bm \Theta^{(m)\top}_{\star}\bm x^{(m)}_i, \eps^{(m)}_i, c_i).
\end{equation}
The equivalent Gaussian inputs $\bg_i = (\bg_i^{(1)}, \dots, \bg_i^{(M)})$ and their associated labels $\by(\bG)$ are defined as in \eqref{eq:gaussian_equivalent_model}.

\paragraph{Statistical metric and free energy ---} Our goal is to study statistical metrics for some function $h: \mathbb{R}^{M \times k \times p} \to \mathbb{R}$ of the form $h(\bTheta^{(1)},\cdots,\bTheta^{(M)})$. For instance, the metric $h$ could be the population risk (a.k.a. generalization error), or some overlap between $\bTheta$ and $\bTheta_\star$.  We define the following coupling free energy function:
\begin{equation}
\label{eq:free_energy_def}
    f_{n,s}(\bTheta[1:M_1],\bm \bX, \by) = -\frac1{n}{\log\int} e^{-sn~h(\bTheta^{(1)}, \dots, \bTheta^{(M)})}dP^{(M_1+1):M},
\end{equation}
where $P^{(M_1+1):M}$ denotes the product measure of the $P_m$ defined in \eqref{def:sampling}.
This gives rise to the following joint objective:
\begin{equation}\label{eq:def:joint_minimization_sampling}
\hR_{n, s}(\bTheta[1:M_1],\bX,\by)=\sum_{m=1}^{M_1}\hR^{(m)}_n(\bTheta^{(m)};\bX^{(m)},\by^{(m)}) + f_{n,s}(\bTheta[1:M_1],\bm \bX, \by).
\end{equation}
In particular, when $s = 0$ we have $f_{n, 0} = 0$ and the problem reduces to the joint minimization problem in \eqref{def:min_multi}. Our first result concerns the universality of the minimum of the above problem:
\begin{proposition}[Universality for joint minimization and sampling]
\label{prop:joint_univ} 
Under Assumptions \ref{assump:clt} For any $s > 0$ and bounded Lipschitz function $\Phi: \dR \to \dR$, and denoting $\hR_{n, s}^{\star}(\bm X, \bm y)\coloneqq \min \hR_{n, s}(\bTheta;\bm X, \bm y)$:
\[\lim_{n, p \to \infty} \left| \Ea{\Phi\left(\hR_{n, s}^\star(\bm X, \bm y(\bm X))\right)} - \Ea{\Phi\left(\hR_{n, s}^\star(\bm G, \bm y(\bm G))\right)} \right| = 0\]
\end{proposition}
The proof is an easy reduction to Theorem \ref{thm:main_gaussian_universality}, and can be found in Appendix \ref{sec:app:joint_univ}.

The next result concerns the value of $h$ at the minimizers point $(\hat{\bTheta}^{(1)}, \dots, \hat{\bTheta}^{(M)})$. We make the following additional assumptions:
\begin{assump}[Differentiable Limit]\label{assump:diff_lim}
There exists a neighborhood of $0$ such that the function
\begin{equation}
    q_{n}(s)=\Ea{\hR_{n, s}^\star(\bm G, \bm y(\bm G))}
\end{equation}
converges pointwise to a function $q(s)$ that is differentiable at $0$.
\end{assump}
For a fixed realization of the dataset $\bX$, we denote by $\lin{h(\bTheta^{(1)},\cdots,\bTheta^{(M)})}_\bX$ the expected value of $h$ when $(\hat{\bTheta}^{(1)},\dots,\hat{\bTheta}^{(M_1)})$ are obtained through the minimization of \eqref{def:min_multi} and $(\bTheta^{(M_1+1)},\dots,\bTheta^{(M)})$ are sampled according to the Boltzmann distributions \eqref{def:sampling}. 

\begin{assump}\label{assump:indep_min}
With high probability on $\bX, \bG$, the value $\lin{h(\bTheta^{(1)},\cdots,\bTheta^{(M)})}_\bX$ (resp. the same for $\bG$) is independent from the chosen minimizers in \eqref{def:min_multi}.
\end{assump}

 In other words, the metric $h$ respects the symmetries of the problem. Then the following holds:
\begin{theorem}\label{thm:overlap_universality} Under assumptions \ref{assump:loss}-\ref{assump:indep_min}, we have:
\begin{equation}
   \lim_{n, p \rightarrow \infty}\left|\Ea{\lin{h\left(\bTheta^{(1)},\cdots,\bTheta^{(M)}\right)}_\bX} - \Ea{\lin{h\left(\bTheta^{(1)},\cdots,\bTheta^{(M)}\right)}_\bG} \right| = 0,
\end{equation}
\end{theorem}

\paragraph{Proof Sketch:}

Our proof relies on the observation that $q_{n}(s)$ is a concave function of $s$. We further have:
    \begin{equation}\label{eq:first_diff}
        q'_{n}(0) = \Ea{\lin{h\left(\bTheta^{(1)},\cdots,\bTheta^{(M)}\right)}_{\bG}}.
    \end{equation}

Subsequently, we utilize a standard result from convex analysis relating the convergence of a sequence of convex or concave functions to the convergence of the corresponding derivatives. The above result shows that the expected value of $h\left(\bTheta^{(1)},\cdots,\bTheta^{(M)}\right)$ for a multi-modal data satisfying the 1d CLT is equivalent to that of a mixture of Gaussians.
It generalizes the universality of generalization error in \cite{montanari2022universality,hu2020universality}to arbitrary function of parameters arising from an arbitrary number of Hamiltonians for mixture models. The full theorem is presented and proven in App. \ref{app:gaussian_universality}. While we prove our result for convergence in expectations and assumption \ref{assump:diff_lim}, the result can be generalized using standard techniques to convergence in probability and other related assumptions. We refer to the alternative assumptions in Theorem 3 of \cite{montanari2022universality} for more details.

\subsection{Universal Weak Convergence}
\label{sec:universal_wc}

Theorem \ref{thm:overlap_universality} provides a general framework for proving the equivalence of arbitrary functions of parameters obtained by minimization/sampling on a given mixture dataset and the equivalent gaussian mixture distribution. 
However, it relies on the assumption of a differentiable limit of the free energy (assumption \ref{assump:diff_lim}). If the assumption holds for a sequences of functions belonging to dense subsets of particular classes of functions, it allows us to prove convergence of minimizers themselves, in a weak sense. We illustrate this through a simple setup considered in \cite{loureiro2021learning_gm}, which precisely characterized the asymptotic distribution of the minimizers of empirical risk with GMM data in the strictly convex case. Consider the following setup:
\begin{equation}\label{eq:gmm_erm}
\left(\hat{\bW}^\bX, \hat{\bbb}^\bX\right) =\ \argmin_{\bW, \bbb}\ \sum_{i=1}^n\ell\left(\frac{\bW\bx_i}{\sqrt d}+\bbb, \by_i\right)+\lambda r(\bW),
\end{equation}
where $\bW \in \dR^{|\cC| \times d}$, $\bbb \in \dR^{|\cC|}$ and $\by_i \in \mathbb{R}^{|\cC|}$ is the one-hot encoding of the class index $c_i$.
For simplicity, we restrict ourselves to diagonal bounded covariances and bounded means.

\begin{assump}\label{assump:cov_bounded}
All of the covariance matrices $\Sigma_c$ are diagonal, with strictly positive eigenvalues $(\sigma_{c, i})_{i \in [d]}$, and there exists a constant $M > 0$ such that for any $c \in \cC$
\begin{equation}
    \sigma_{c, i} \leq M \quand \norm{\bmu_c}_2 \leq M.
\end{equation}
\end{assump}

Secondly, since we aim at obtaining a result on the weak convergence of the estimators, we assume the same weak convergence for the means and covariances, and that the regularization only depends on the empirical measure of $\bW$.
\begin{assump}\label{assump:emp_conv}
    The empirical distribution of the $\bmu_c$ and $\bSigma_c$ converges weakly as follows:
    \begin{equation}
        \frac1d \sum_{i=1}^d \prod_{c \in \cC} \delta(\mu_c - \sqrt{d}\mu_{c, i}) \delta(\sigma_c - \sigma_{c, i}) \quad \xrightarrow[d \to \infty]{\cL} \quad p(\bm\sigma, \bm\mu)
    \end{equation}
\end{assump}
\begin{assump}
\label{assump:reg}
The regularizer $r(\cdot)$ is of the following form:
\begin{equation}
    r(\bW) = \sum_{i=1}^d \psi_r(\bW_i),
\end{equation}
for some convex and twice differentiable function $\psi_r:\dR \rightarrow \dR$.
\end{assump}

Under these conditions, the joint measure of the minimizers and of the data moments converges weakly to a fixed limit, independent of the data-distribution:
\begin{theorem}\label{thm:weak_conv}
Assume that conditions \ref{assump:loss}-\ref{assump:reg} hold, and further that the function $\ell(\bullet,y)+r(\bullet)$  is convex and coercive.
Then, for any bounded-Lipschitz function: $\Phi:\R^{3|\cC|} \rightarrow \R$, we have:
\begin{equation}\label{eq:joint_emp_measure}
    \Ea{\frac{1}{d}\sum_{i=1}^d \Phi(\{(\hat{\bW}^\bX)_{c, i}\}_{c \in \cC},\{\mu_{c, i}\}_{c \in \cC},\{\sigma_{c, i}\}_{c \in \cC})} \xrightarrow{d\to+\infty}\Eb{\tilde p}{\Phi(\bw, \bm\mu, \bm\sigma)},
\end{equation}
where $\tilde p$ is a measure on $\dR^{3|\cC|}$, that is determined by the so-called \emph{replica equations}.
\end{theorem}

\paragraph{Proof Sketch}
Let $\bTheta$ denote the combined set of parameters $\{\bW,\bbb\}$. We first show that for $h(\bTheta)$ having bounded second derivatives, the
perturbation term $sh(\bTheta)$ can be absorbed into the regularizer, while maintaining the assumptions in \cite{loureiro2021learning_gm}. 
Next, we show that the sequence of solutions converge in a suitable sense, and are described through the limits of the replica equations from \cite{loureiro2021learning_gm}. This allows us to prove that assumption \ref{assump:diff_lim} holds for functions that can be expressed as expectations w.r.t the joint empirical measure in \ref{eq:joint_emp_measure}. More details can be found in Appendix \ref{app:weak_conv}.

\subsection{One-dimensional CLT for Random Features} \label{sec:1d-CLT}
We finally show a conditional one-dimensional CLT for a random features map applied to a mixture of gaussians, in the vein of those shown in \cite{goldt2020gaussian, hu2020universality, montanari2022universality}. Concretely, we consider the following setup:
\begin{equation}
    \bx_i = \sigma(\bF^\top \bz_i), \quad \bz_i \sim \sum_{c \in \cC} \mathcal{N}(\bmu_c^\bz, \bSigma_c^\bz),
\end{equation}
where the feature matrix $\bF \in \dR^{d \times p}$ has i.i.d $\cN(0, \sfrac1d)$ entries. This setup is much more permissive than the ones in \cite{hu2020universality, montanari2022universality}, that restrict the samples $\bz$ to standard normal vectors. However, we do require some technical assumptions:
\begin{assump}\label{assump:1dclt_bounded}
    The activation function $\sigma$ is thrice differentiable, with $\norm{\sigma^{(i)}}\leq M$ for some $M > 0$, and we have
    \begin{equation}
        \Eb{g \sim \cN(0, 1)}{\sigma(g)} = 0.
    \end{equation}
    The cluster means and covariances of $\bz$ satisfy for all $c \in \cC$
    \begin{equation}
        \norm{\bmu_c^\bz} \leq M, \qquad \norm{\bSigma_c^\bz}_{\mathrm{op}} \leq M
    \end{equation}
    for some constant $M > 0$.
\end{assump}
 We also place ourselves in the proportional regime, i.e. a regime where $p / d \in [\gamma^{-1}, \gamma]$ for some $\gamma > 0$. For simplicity, we will consider the case $k = 1$; and the constraint set $\cS_p$ as follows:
 \begin{equation}\label{eq:1d_clt_constraints}
     \cS_p = \Set*{\btheta \in \dR^d \given \norm{\btheta}_2 \leq R, \quad \norm{\btheta}_\infty \leq Cp^{-\eta}}
 \end{equation}
 for a given $\eta > 0$. We show in the appendix the following theorem:
\begin{theorem}\label{thm:1d_clt}
    Under Assumption \ref{assump:1dclt_bounded}, and with high probability on the feature matrix $\bF$, the data $\bX$ satisfy the concentration assumption \ref{assump:concentration}, as well as the one-dimensional CLT of Assumption \ref{assump:clt}. Consequently, the results of Theorems \ref{thm:single_univ_free} and \ref{thm:overlap_universality} apply to $\bX$ and their Gaussian equivalent $\bG$.
\end{theorem}

\paragraph{Proof Sketch}
Our proof constructs a reduction to the one-dimensional CLT for random features in \cite{hu2020universality,montanari2022universality,goldt2020gaussian}. We first note that the one-dimensional CLT in \cite{hu2020universality,montanari2022universality,goldt2020gaussian} holds even when the activation functions differ across neurons.
Subsequently, we proceed by defining the following neuron-wise activation functions:
\begin{equation}
    \sigma_{i,c}(u) =  \sigma(u+\bm{f}^\top_i\bmu_{c}).
\end{equation}
However, the above activation functions depend on the means and the dimensions of the inputs. Our proof involves controlling the effect of this dependence.
Additionally, we handle the effect of the covariance by observing that the features $\bx$ can be equivalently generated by applying the modified weights $\bSigma_c^{\bz 1/2}\bF$ to isotropic Gaussian noise. Further details can be found in Appendix \ref{app:finite_1dclt}.

While we prove the above result for random weights, we note, however that the non-asymptotic results in \cite{hu2020universality, goldt2020gaussian} also hold for deterministic matrices satisfying approximate orthogonality conditions. Therefore, we expect the one-dimensional CLT to approximately hold for a much larger class of feature maps. Finally, we also note that the above extension of the one-dimensional CLT to mixture of gaussians also provides a proof for the asymptotic error for random features in \cite{refinetti2021classifying}.

\section{Conclusion}  
We demonstrate the universality of the Gaussian mixture assumption in high-dimension for various machine learning tasks such as empirical risk minimization, sampling and ensembling, in a variety of settings including random features or GAN generated data. We also show that universality holds for a large class of functions,  and provide a weak convergence theorem. These results, we believe, vindicate the classical theoretical line of works on the Gaussian mixture design. We hope that our results will stimulate further research in this area.

\section*{Acknowledgements}
We acknowledge funding from the ERC under the European Union’s Horizon 2020 Research and Innovation Program Grant Agreement 714608-SMiLe, the Swiss National Science Foundation grant SNFS OperaGOST, $200021\_200390$ and the \textit{Choose France - CNRS AI Rising Talents} program.

\newpage

\bibliographystyle{plainnat}
\bibliography{bib}
\clearpage

\newpage
\appendix
\section{Proofs of Main Results}\label{app:gaussian_universality}
\renewcommand{\theassump}{A\arabic{assump}}
\setcounter{assump}{0}

\subsection{Notation}
We follow the setting defined in section \ref{sec:set}. Throughout, we work in the so-called proportional high-dimensional limit, where $n, p$ go to infinity with
\[\frac np \to \alpha > 0, \]
while $\cC$ stays fixed.

Throughout this section, $\norm{}$ will denote the spectral norm of a matrix, while $\norm{}_q$ for $q > 0$ will refer to the element-wise $q$-norms. For a subgaussian random variable $Y$, its sub-gaussian norm $\norm{Y}_{\psi_2}$ is defined as
\[ \norm{Y}_{\psi_2} = \inf \Set*{t > 0 \given \Ea{\exp\left(\frac{Y^2}{t^2}\right)} \leq 2}. \]

\addtocontents{toc}{\protect\setcounter{tocdepth}{1}}
\subsection{State of the art}
\addtocontents{toc}{\protect\setcounter{tocdepth}{2}}

There have been many recent progress on Gaussian-type low-dimensional CLT and universality recently \cite{goldt2020gaussian,hu2020universality,montanari2022universality}. We shall leverage on these results to prove our first theorem.

In particular, the starting point for our mathematical proof will use the recent result of \cite{montanari2022universality} which we shall now review. Consider the minimization problem
\eqref{def:min_problem_committee}, with $(\bm x_\mu, y_\mu)$ i.i.d random variables; the goal is to replace the $\bm x_\mu$ by their Gaussian equivalent model
\begin{equation} 
    \bm g_i \overset{i.i.d}{\sim} \cN(\bmu, \bSigma) \qquad \text{where}\qquad\bmu = \Ea{\bx},\quad \bSigma = \Ea{\bm x\bm x^\top}.
\end{equation}
\cite{montanari2022universality} make the following assumptions:

\begin{assump}[Loss and regularization]\label{assump:app:loss}
    The loss function $\ell : \dR^{k+1} \to \dR$ is nonnegative and Lipschitz, and the regularization function $r : \dR^{p\times k} \to \dR$ is locally Lipschitz, with constants independent from $p$.
\end{assump}

\begin{assump}[Concentration on the directions of $\cS_p$]\label{assump:app:concentration}
    We have
    \begin{equation}
    \sup_{\bm \theta \in \cS_p, \norm{\bm \theta}_2 \leq 1} \norm{\bm \theta^\top \bm x}_{\psi_2} \leq M, \quad \sup_{\bm \theta \in \cS_p, \norm{\bm \theta}_2 \leq 1} \norm{\bSigma^{1/2}\bm \theta}_{2} \leq M, \quand \norm{\bmu}_2 \leq M
    \end{equation}
    for some constant $M > 0$.
\end{assump}

\begin{assump}[One-dimensional CLT]\label{assump:app:clt}
    For any bounded Lipschitz function $\varphi: \dR^k \to \dR$,
    \begin{equation}
        \lim_{p\to \infty} \sup_{\btheta\in\cS_p} \left\lvert\Ea{\phi(\btheta^\top \bm x)} - \Ea{\phi(\btheta^\top \bm g)}\right\rvert|=0.
    \end{equation}
\end{assump}

\begin{assump}[Labels]\label{assump:app:labels}
    The $y_\mu$ are generated according to
    \begin{equation}\label{eq:label_definition_montanari}
    y_i = \eta(\bm \Theta^* \bm x_i, \eps_i, c_i),
    \end{equation}
    where $\eta : \dR^{k^*+1}\to \dR$ is a Lipschitz function, $\bm \Theta^* \in \cS_p^{k^*}$, and the $\eps_\mu$ are i.i.d subgaussian random variables with
    \[ \norm{\eps_i}_{\psi_2} \leq M \]
    for some constant $M > 0$.
\end{assump}

Building on those assumptions, \cite{montanari2022universality} prove the following:
\begin{theorem}[Theorem 1. in \cite{montanari2022universality}] \label{thm:montanari_universality}
Suppose that Assumptions \ref{assump:app:loss}-\ref{assump:app:clt} hold. Then, for any bounded Lipschitz function $\Phi: \dR \to \dR$, we have
\[\lim_{n, p \to \infty} \left| \Ea{\Phi\left(\widehat\cR_n^\star(\bm X, \bm y(\bm X))\right)} - \Ea{\Phi\left(\widehat\cR_n^\star(\bm G, \bm y(\bm G))\right)} \right| = 0\]
In particular, for any $\rho \in \dR$,
\[ \widehat\cR_n^\star(\bm X, \bm y(\bm X)) \overset{\dP}{\longrightarrow} \rho \quad \text{if and only if} \quad \widehat\cR_n^\star(\bm G, \bm y(\bm G)) \overset{\dP}{\longrightarrow} \rho \]
\end{theorem}

\paragraph{Free energy approximation} 

A crucial component of the proof in \cite{montanari2022universality} is the approximation of the minimizer through a free energy function. 
Define the discretized free energy
\begin{equation}
    f_{\epsilon, \beta}(\bm X) = \frac1{n\beta}\sum_{\bm\Theta\in\cN_\epsilon^k} \exp\left( -\beta\,  \widehat\cR_n(\bm\Theta ; \bm X, \bm y(\bm X) ) \right),
\end{equation}
where $\cN_\epsilon$ is a minimal $\epsilon$-net of $\cS_p$.
\begin{lemma}[Lemma 1 in \cite{montanari2022universality}]\label{lem:montanari_free_energy}
 For any bounded differentiable function $\Phi$ with bounded Lipschitz derivative, and any $\epsilon >0$ we have:
\begin{equation}
\nonumber
 \lim_{n, p\to\infty} \left|\Ea{\Phi\left(f_{\epsilon,\beta}(\bX)\right)} - 
\Ea{\Phi\left(f_{\epsilon,\beta}(\bG)\right)}\right| = 0.
\end{equation}
\end{lemma}

Subsequently, using classical arguments from both the theory of $\epsilon$-nets and statistical physics, the authors show that
\begin{equation}\label{eq:free_energy_approx}
    \left|f_{\epsilon, \beta}(\bm X) - \widehat\cR_n(\bm\Theta ; \bm X, \bm y(\bm X) \right| \leq C_1(\epsilon) + \frac{C_2(\epsilon)}{\beta},
\end{equation}
and the same inequality holds for $\bm G$. Since $C_1, C_2$ do not depend on $n, p$, it is possible to choose first $\epsilon$, then $\beta$ so that the RHS of \eqref{eq:free_energy_approx} is as small as desired.

\cite{montanari2022universality} therefore used the universality of the free energy as an intermediate step wards proving the universality of the training error.
We generalize this result to the free energy defined for a general class of Boltzmann distributions, allowing us to prove universality in applications related to sampling.

\subsection{Sketch of proof of Lemma \ref{lem:montanari_free_energy}, adapted from \cite{montanari2022universality}}

\paragraph{Interpolation path}
For any $0 \leq t \leq \pi/2$, define
\[ \bm U_t = \cos(t)\bm X + \sin(t) \bm G\]
Then $\bm U_t$ is a smooth interpolation path with independent columns, ranging from $\bm U_0 = \bm X$ to $\bm U_{\pi/2} = \bm G$. We can write, for any differentiable function $\psi$,
\[ \left|\Ea{\psi(f_{\epsilon, \beta}(\bm X))} - \Ea{\psi(f_{\epsilon, \beta}(\bm G))} \right| \leq \int_0^{\pi/2} \left|\Ea{\frac{d \psi(f_{\epsilon, \beta}(\bm U_t)))}{dt}}\right|  dt,\]
and by the dominated convergence theorem it suffices to show that the integrand converges to $0$ for any $t$. The chain rule gives:
\begin{equation}\label{eq:chain_interp} 
\frac{d \psi(f_{\epsilon, \beta}(\bm U_t)))}{dt} = \psi'(f_{\epsilon, \beta}(\bm U_t)))\left( \sum_{\mu=1}^n \left(\frac{d\bm u_{t, \mu}}{dt}\right)^\top \nabla_{\bm u_{t, \mu}} f_{\epsilon, \beta}(\bm U_t) \right),
\end{equation}

and the dependency in $\psi$ can be easily controlled. Since all columns of $\bm U_t$ are i.i.d, we are left with showing
\begin{equation}\label{eq:interp_universality}
    \lim_{n, p \to \infty} n \dE_{(1)}\left[\left(\frac{d\bm u_{t, 1}}{dt}\right)^\top \nabla_{\bm u_{t, 1}} f_{\epsilon, \beta}(\bm U_t)\right] = 0 \quad \text{a.s.},
\end{equation}
where $\dE_{(1)}$ denotes the expectation with respect to $(\bm x_1, \bm g_1, \eps_1)$.

\paragraph{Showing \eqref{eq:interp_universality}}

Imagine for a moment that $\bm x_1$ is Gaussian; then $\bm u_{t, 1}$ and $d\bm u_{t, 1}/dt$ are also jointly Gaussian, and we have
\begin{align*}
\Ea{\left(\frac{d\bm u_{t, 1}}{dt}\right)^\top \bm u_{t, 1}} &= \Ea{(-\sin(t)\bm x_1 + \cos(t)\bm g_1)^\top (\cos(t)\bm x_1 + \sin(t)\bm g_1)} \\
&= 0,
\end{align*}
since $x_1$ and $g_1$ have the same covariance by definition. Therefore, they are independent, and we have
\begin{align*}
    \dE_{(1)}\left[\left(\frac{d\bm u_{t, 1}}{dt}\right)^\top \nabla_{\bm u_{t, 1}} f_{\epsilon, \beta}(\bm U_t)\right] = \dE_{(1)}\left[\left(\frac{d\bm u_{t, 1}}{dt}\right)\right]^\top \dE_{(1)}\left[\nabla_{\bm u_{t, 1}} f_{\epsilon, \beta}(\bm U_t)\right] = 0.
\end{align*}
On the other hand, it is possible to show that $\bm x_1$ only appears in \eqref{eq:interp_universality} through scalar products with $\bm \Theta$ or $\bm \Theta^*$. As a result, we can leverage Assumption $\ref{assump:clt}$ to replace $\bm x_1$ by a Gaussian vector $\bm w$ independent from $\bm g_1$ as $p \to \infty$. Then, the reasoning above can be repeated with $\bm w$ and $\bm g_1$ to conclude the proof.

\subsection{Proof of Theorem \ref{thm:main_gaussian_universality}}
In order to prove our theorem \ref{thm:main_gaussian_universality}, we now aim to adapt the proof from \cite{montanari2022universality} to the following case where the distribution of $x$ can be a {\it mixture} of several other distributions, each with different mean and covariance. For a discrete set $\cC$, we consider a family of distributions $(\nu_c)_{c\in \cC}$ on $\dR^p$, with means and covariances
\[ \bmu_c = \dE_{\zz\sim \nu_c}[\zz] \qquand \bSigma_c = \dE_{\zz\sim \nu_c}[\zz \zz^\top] \]
Given a type assignment $\sigma: [n] \to \cC$, each sample $x_i$ is then drawn independently from $\nu_{\sigma(i)}$. The equivalent Gaussian model is straightforward: we simply take
\[ \bm g_i \sim \cN(\bmu_{\sigma(i)}, \bSigma_{\sigma(i)}), \]
independently from each other. An important special case of this setting is when $\sigma$ is itself random, independently from the $\bm x_i$ and $\bm g_i$: the law of $\bm g_i$ is then a so-called Gaussian Mixture Model. Note that in the Gaussian mixture setting, the existence of the labeling function $\sigma$ implies that we coupled the labels for $\bX$ and $\bG$.

The main differences between our Assumptions \ref{assump:loss}-\ref{assump:clt} and Assumptions \ref{assump:app:loss}-\ref{assump:app:clt} are the following:
\begin{enumerate}
    \item Assumption \ref{assump:loss} is unchanged.
    \item We relax \eqref{eq:label_definition_montanari} in Assumption \ref{assump:labels} into 
    \[ y_i = \eta_{\sigma(i)}(\bm \Theta^* \bm x_i, \eps_i, c_i), \]
    when $\eta$ a Lipschitz function in its first two parameters. This allows in particular to incorporate classification problems in our setting, at no cost in the proof complexity.
      \item We assume a more general setup where the constraint set $\cS_p$ is not necessarily a product set. This slight generalization will be useful while proving a reduction to multiple objectives in Theorem \ref{thm:overlap_universality}. 
    
    \item We suppose that Assumptions \ref{assump:concentration} and \ref{assump:clt} hold for any possible distribution $\nu_c$ for $c \in \cC$ and its associated Gaussian equivalent model.
    \item We allow the reference measures to be any sequence of Borel measures with support on $\cS_p$, instead of only the Dirac measure on the $\epsilon$-net $\cN_\epsilon$.
\end{enumerate}
We now go through the proof of the previous section, highlighting the important changes.

\paragraph{Free energy approximation} This section goes basically unchanged; the approximation between $\widehat\cR_n^*(\bm X, \bm y(\bm X))$ and $f_{\epsilon, \beta}(\bm X)$ relies on Lipschitz arguments and concentration bounds on the $\bm x_i$ and $\bm g_i$, which are satisfied by our modification of Assumption \ref{assump:concentration}.

\paragraph{General Reference Measures} Our proof for the universality of free energy in Theorem \ref{thm:main_gaussian_universality} is a generalization of the proof of Lemma 1 in \cite{montanari2022universality}. 
Compactness of the supports ensures that the corresponding free energy:
\begin{equation}
     f_{\beta,n}(\bZ)= \int \exp\left(-\beta nR_n(\bTheta;\bZ,\by)\right)d\mu(\bTheta),
\end{equation}
is finite for any $\bZ$.

The corresponding Boltzmann measures can then be defined by 
setting the Radon-Nikodym derivative/density to be:
\begin{equation}
\frac{d\tilde{\mu}}{d\mu}=\exp\left(-\beta n\hR_n(\bTheta;\bX,\by)\right).
\end{equation}
The measure $\tilde{\mu}$ is then a Borel measure with support lying in $\cS_p$. Therefore, through dominated convergence theorem, we can interchange differentiation and expectations w.r.t $\mu$ in the proof of Lemma 1 in \cite{montanari2022universality}.
For instance, equation \eqref{eq:chain_interp} can be expressed as:
\begin{align}
\label{eq:derivative_explicit}
 \E\left[\frac{\partial}{\partial t} 
\psi(f_{\beta,n}(\bU_t))\right]
&= \E\left[\frac{\psi'(f_{\beta,n}(\bU_t))}{n}  \sum_{i=1}^n 
\frac{
\int\widetilde\bu_{t,i}^\top
 \widehat\bd_{t,i}(\bTheta) 
e^{-n\beta R_n(\bTheta;\bZ,\by)}d\mu_p}
{\int e^{-n\beta R_n(\bTheta;\bZ,\by)}d\mu_p}\right].
\end{align}

Similarly we substitute $\sum_{\bTheta}$ by $\int d\mu$ in the remaining arguments in the proof of Lemma 1 in \cite{montanari2022universality}.

\paragraph{Interpolation path} Recall that the important property of $\bm U_t$ is that
\begin{equation}\label{eq:ut_constant_norm}
\Ea{\left(\frac{d\bm U_t}{dt}\right)^\top \bm U_t} = 0.
\end{equation}
To this end, we set
\[ \bm u_{t, i} = \bmu_{\sigma(i)} + \cos(t)(\bm x_i - \bmu_{\sigma(i)}) + \sin(t)(\bm g_i - \bmu_{\sigma(i)}), \]
and it is easy to check that \eqref{eq:ut_constant_norm} is satisfied. Another problem is that the columns of $\bm U_t$ are not i.i.d anymore, so we have to control
\begin{equation}\label{eq:interp_non_iid}
    \frac 1n \sum_{i=1}^n \left| \dE_{(i)}\left[\left(\frac{d\bm u_{t, i}}{dt}\right)^\top \nabla_{\bm u_{t, i}} f_{\epsilon, \beta}(\bm U_t)\right] \right|,
\end{equation}
where this time $\dE_{(i)}$ is the expectation w.r.t $(\bm x_i, \bm g_i, \eps_i)$. However, \eqref{eq:interp_non_iid} is a weighted average over all values of $\sigma(\mu)$, and since $\cC$ is finite is suffices to show \eqref{eq:interp_universality} for any value of $\sigma(1)$.

\paragraph{Showing \eqref{eq:interp_universality}} This section again relies on concentration properties of the $\bm x_i$ and $\bm g_i$, as well as Assumption \ref{assump:clt}. The arguments thus translate directly from \cite{montanari2022universality}.



\subsection{Proof of Proposition \ref{prop:joint_univ}}
\label{sec:app:joint_univ}

We define the following free energy of the system:
\begin{equation}\label{eq:app:joint_free_energy}
    f_{n,s, \epsilon}(\bm \bX, \by) = -\frac1{n}{\log\int} e^{-n \sum_{m=1}^M \beta_m\,  \widehat\cR_n^{(m)}(\bm\Theta^{(m)} ; \bm \bX^{(m)}, \bm y^{(m)}) -sn~h(\bTheta^{(1)}, \dots, \bTheta^{(M)})}d\mu_{\epsilon}^{1:M_1} d\mu^{(M_1+1):M},
\end{equation}
where the $\mu_m$ are the reference measures for the Boltzmann distributions in \eqref{def:sampling}, and $\mu_{\epsilon}^m$ is the uniform measure supported on a minimal $\epsilon$-net of $\cS^{(i)}_p$:
\begin{equation}
    \mu^m_\epsilon = \frac{1}{\abs{\cN^m_\epsilon}}\sum_{\bTheta \in \cN^m_\epsilon}\delta_{\bTheta},
\end{equation}
where $\delta_{\bTheta}$ denotes the Dirac measure at $\bTheta$. We establish the following result:
\begin{lemma}[Universality of the joint free energy]
Under Assumptions \ref{assump:loss}-\ref{assump:clt}, for any fixed $\epsilon >0$ and any bounded differentiable function $\psi$ with bounded Lipschitz derivative we have.
  \begin{equation}
\nonumber
 \lim_{n\to\infty} \left|\Ea{\psi(f_{n,s, \epsilon}(\bm \bX, \by))} - 
\Ea{\psi(f_{n,s, \epsilon}(\bm \bX, \by))}\right| = 0.
\end{equation}  
\end{lemma}

\begin{proof}
    We construct a reduction from the free energy of the form in equation \ref{eq:app:joint_free_energy} to the universality of free energy for a single objective in Theorem \ref{thm:single_univ_free}. 
    
We construct an equivalent objective on the $pM\times kM$ dimensional space. Consider the following mapping:
\begin{equation}\label{eq:mat_comb}
    \bTheta = \begin{pmatrix}
        \bTheta^{(1)} & \bm 0 & \bm 0 & \cdots\\
        \bm 0 & \bTheta^{(2)} & \bm 0 & \cdots\\
        \cdots & \cdots & \cdots &  \cdots\\
        \bm 0 & \bm 0& \cdots &\bTheta^{(M)}
    \end{pmatrix}
\end{equation}.
Let $\cS_{pM}$ be the set obtained by applying the mapping in equation \ref{eq:mat_comb} to $\cS^{(1)}_p,\cdots,\cS^{(M)}_p$. 
Let $\bX = (\bX^{(1)},\cdots, \bX^{(M)})$ denote the combined input matrix with each row of dimension $pM$. We note that $\cS_{pM}$ is a product of $kM$ compact sets, each satisfying the assumption \ref{assump:concentration}.
Similarly, we have the combined output vector:
\begin{equation}
    \by = \begin{pmatrix}
        \by^{(1)}\\
        \vdots\\
        \by^{(M)}
    \end{pmatrix}
\end{equation}

We define $\ell:\R^{k'M}\times \dR^{kM}  \rightarrow \dR$ by:
\begin{equation}
\ell(\bu,\yy)=\sum_{m=1}^M \beta_m \ell_m(\bu[(k-1)m:km],\yy[(k-1)m:km]).
\end{equation}
Similarly, we define the total regularization as:
\begin{equation}
\begin{split}
    r(\bTheta)&= \sum_{m=1}^M \beta_m r_m(\bTheta[(m-1)p:mp,(m-1):k,mk])\\ &+ sh(\bTheta[0:p,0:k],\cdots,\bTheta[(M-1)p:Mp,(M-1)k:Mk]).
\end{split}
\end{equation}
Let $\hR_n(\bTheta;\bX,\by)$ denote the following objective on the combined vector $\bTheta$:
\begin{equation}
    \hR_n(\bTheta;\bX,\by) = \frac{1}{n} \sum_{i=1}^n \ell(\bTheta^\top \bx_i,\yy_i)+ r(\bTheta)
\end{equation}
Then, using all the definitions above, we have
\begin{equation}
\hR_n(\bTheta;\bZ,\by)=\sum_{m=1}^M \beta_m\,  \widehat\cR_n^{(m)}(\bm\Theta^{(m)} ; \bm \bX^{(m)}, \bm y^{(m)} )+s~h(\bTheta^{(1)},\cdots,\bTheta^{(M)})
\end{equation}
Therefore, we obtain that:
\begin{equation}
\begin{split}
     f_{n,s, \epsilon}(\bm \bX, \by) &= -\frac1{n}{\log\int} e^{-n \sum_{m=1}^M \beta_m\,  \widehat\cR_n^{(m)}(\bm\Theta^{(m)} ; \bm \bX^{(m)}, \bm y^{(m)}) -sn~h(\bTheta^{(1)}, \dots, \bTheta^{(M)})}d\mu_{\epsilon}^{1:M_1} d\mu^{(M_1+1):M}\\
    &=-\frac1{n}{\log\int} e^{-n\hR_n(\bTheta;\bX,\by)}d\mu_{\epsilon}^{1:M_1}d\mu^{(M_1+1):M}.
\end{split}
\end{equation}
We further note that the constraint sets on $\bTheta$, and the joint means, covariances on $\bm X$ satisfy the assumptions \ref{assump:app:concentration}. Therefore, using Theorem \ref{thm:single_univ_free}, we obtain that:
 \begin{equation}
\nonumber
 \lim_{n\to\infty} \left|\Ea{\psi(f_{n,s, \epsilon}(\bm \bX, \by))} - 
\Ea{\psi(f_{n,s, \epsilon}(\bm \bX, \by))}\right| = 0
\end{equation}  
\end{proof}
The proof of Proposition \ref{prop:joint_univ} then follows using the $\epsilon$-net approximation in \eqref{eq:free_energy_approx} for $\bTheta[1:M_1]$.

\subsection{Proof of Theorem \ref{thm:overlap_universality}}

Our proof relies on the following result:
\begin{lemma}\label{lem:app:convex_conv}
For any $n \in \mathbb{N}$, $q_{n}(s)$ is a concave function of $s$, that is differentiable at $0$, and
\begin{equation}
    q'_{n}(0) = \Ea{\lin{h\left(\bTheta^{(1)},\cdots,\bTheta^{(M)}\right)}_{\bG}}.
\end{equation}
\end{lemma}
\begin{proof}
    Consider the joint free energy in Equation \eqref{eq:free_energy_def}:
\begin{equation}
    f_{n,s}(\bTheta[1:M_1],\bm \bX, \by) = -\frac1{n}{\log\int} e^{-sn~h(\bTheta^{(1)}, \dots, \bTheta^{(M)})}dP^{(M_1+1):M},
\end{equation}
Let $\langle\cdot\rangle_{M_1+1:M}$ denote the expectation w.r.t the product measure
\[ d \tilde \mu(\bTheta[1:M_1],\bm \bX, \by) =  e^{-sn~h(\bTheta^{(1)}, \dots, \bTheta^{(M)})}dP^{(M_1+1):M}. \]
We observe that:
\begin{equation}\label{eq:app:free_energy_derivative}
    \frac{d f_{n, s}}{ds}(\bTheta[1:M_1],\bm \bX, \by) =\langle h(\bTheta^{(1)}, \dots, \bTheta^{(M)})\rangle_{M_1+1:M}.
\end{equation}
Differentiating w.r.t s again, and using the dominated convergence theorem, we obtain:
\begin{equation}
    -\frac1n\frac{d^2 f_{n, s}}{ds^2}(\bTheta[1:M_1],\bm \bX, \by) =\langle h(\bTheta^{(1)}, \dots, \bTheta^{(M)})^2\rangle_{M_1+1:M}-\langle h(\bTheta^{(1)}, \dots, \bTheta^{(M)})\rangle^2_{M_1+1:M}.
\end{equation}
Since the R.H.S equals the variance of the variable $h(\bTheta^{(1)}, \dots, \bTheta^{(M)})$ w.r.t $\tilde{\mu}^{(M_1+1):M}$, we have:
\begin{equation}
    \frac{d^2 f_{n, s}(\bTheta[1:M_1],\bm \bZ,s)}{ds^2} \leq 0.
\end{equation}
Therefore, for fixed $\bTheta[1:M_1]$, we obtain that the function:
\begin{equation}
\hR_{n, s}(\bTheta[1:M_1],\bX,\by)=\sum_{m=1}^{M_1}\hR^{(m)}_n(\bTheta^{(m)};\bX^{(m)},\by^{(m)}) + f_{n,s}(\bTheta[1:M_1],\bm \bX, \by).
\end{equation}
is concave in s. Next, we recall that pointwise infimum of arbitrary collections of concave functions is concave \citep{rockafellar1970convex}. Therefore, we obtain that the function:
\begin{equation}
    q_{n}(s) = \Ea{\min_{\bTheta[1:M_1]}\hR_{n, s}(\bTheta[1:M_1],\bG,\by)},
\end{equation}
is concave in $s$.
Then, by Danskin's theorem, the subdifferential of $q_{n}$ at zero is the set
\[ \Set*{\langle h(\bTheta^{(1)}, \dots, \bTheta^{(M)})^2\rangle_{M_1+1:M}, \bTheta[1:M_1] \in \argmin \hR_{n, 0}(\bTheta[1:M_1],\bX,\by)} \]
But by Assumption \ref{assump:indep_min}, this set only has one element, and hence $q_n$ is differentiable at $0$.
\end{proof}

Next, we relate the convergence of the above functions to the expectation of $h(\bTheta^{(1)}, \dots, \bTheta^{(M)})$, through the following standard result from Convex Analysis:

\begin{theorem}\label{thm:conv}
(Theorem 25.7. in \cite{rockafellar1970convex}): Let $C$ be an open convex set, and let $f$ be a convex function which is finite and differentiable on $C$. Let $f_1, f_2, \ldots$, be a sequence of convex functions finite and differentiable on $C$ such that $\lim _{n \rightarrow \infty} f_n(x)=f(x)$ for every $x \in C$. Then
$$
\lim _{n \rightarrow \infty} \nabla f_n(x)=\nabla f(x), \quad \forall x \in C .
$$ 
\end{theorem}
By Assumption \ref{assump:diff_lim},
\begin{equation}\label{eq:assump}
    \lim_{n \rightarrow \infty} q_{g,n}(s) = q(s).
\end{equation}
Applying theorem \ref{thm:conv} to the sequence $q_{g,n}(s)$ yields:
\begin{equation}
    \lim_{n \rightarrow \infty} q'_{n}(0)=q'(0).
\end{equation}
Now, consider the corresponding free energy for the data distribution $p_{\bm x}$:
\begin{equation}
q_{x,n}(s) =\Ea{\min_{\bTheta[1:M_1]}\hR_{n, s}(\bTheta[1:M_1],\bX,\by)}.
\end{equation}
We again have that $q_{x,n}(s)$ is a concave, differentiable function in $s$, with:
\begin{equation}\label{eq:grad_x}
q'_{x, n}(0) = \Ea{\lin{h\left(\bTheta^{(1)},\cdots,\bTheta^{(M)}\right)}_{\bX}}.
\end{equation}
Now, Proposition \ref{prop:joint_univ} and equation \eqref{eq:assump} imply that 
 \begin{equation}
   \lim_{n \rightarrow \infty} q_{x,n}(s)= \lim_{n \rightarrow \infty} q_{n}(s) = q(s).
\end{equation}
Therefore, Theorem \ref{thm:conv} applied to the sequence of functions $q_{x,n}(s)$ implies that:
\begin{equation}
    \lim_{n \rightarrow \infty} q'_{x,n}(0)=q'(0).
\end{equation}
By equations \ref{eq:grad_x} and Lemma \ref{lem:app:convex_conv}, we then obtain:
\begin{equation}
   \lim_{n, p \rightarrow \infty}\left|\Ea{\lin{h\left(\bTheta^{(1)},\cdots,\bTheta^{(M)}\right)}_\bX} - \Ea{\lin{h\left(\bTheta^{(1)},\cdots,\bTheta^{(M)}\right)}_\bG} \right| = 0.
\end{equation}

\subsection{Proof of Theorem \ref{thm:1d_clt}: One-dimensional CLT for Random Feature Models }\label{app:finite_1dclt}

Our proof relies on a reduction to Theorem 2 in \cite{hu2020universality} and the proof of corollary 2 in \cite{montanari2022universality}. We first notice that it suffices to show the result when $\bz \sim \cN(\bmu^\bz, \bSigma^\bz)$ is a Gaussian variable; and upon rescaling of $\sigma$ we shall assume that $\tr(\bSigma^\bz) = p$.

Let $\bV = \bSigma^{\bz 1/2} \bF$. We express $\bz$ as $\bz=\bmu^\bz+\bSigma^{\bz 1/2}\bz'$ where $\bz \sim \cN(\bm 0, \bm I)$.
Therefore, we have:
\begin{equation}
   \bx = \sigma(\bF^\top \bz) = \sigma(\bF^\top\bmu^\bz+\bV^\top\bz') .
\end{equation}

We define the following events:
\begin{align*}
    \cA_1 &= \Set*{ \sup_{i, j \in [d]} \left| \bv_i^\top \bv_j - \delta_{ij} \right| \leq C_1\left(\frac{\log d}{d}\right)^{1/2}}
    & \cA_2 &= \Set*{ \sum_{i\in[d]} \left| \norm{\bv_i}^2 - 1 \right| \leq  C_2 } \\
    \cA_3 &= \Set*{\norm{\bF}_{\mathrm{op}} \leq C_3} & \cA_4 &= \Set*{\norm{\bV}_{\mathrm{op}} \leq C_4}
\end{align*}
Since the $\bff_i$ are independent and sub-gaussian, Lemma 22 in \cite{montanari2022universality} implies that there exists constants $C_1, C_2, C_3, C_4$ such that $\cB = \cA_1 \cap \cA_2 \cap \cA_3 \cap \cA_4$ is a high-probability event. Now, for $i \in [d]$, we define
\begin{equation}
    \sigma_i(u) = \sigma(u + \bff_i^\top \bmu^\bz).
\end{equation}
Now, as in \cite{montanari2022universality}, we argue that the proof of Theorem 2 in \cite{hu2020universality} still applies to our setting. Indeed:
\begin{itemize}
    \item since $\bz$ does not have identity covariance, we replace the conditions on $\bF$ by the exact same ones on $\bV$.
    \item the Stein's method they use proceeds term by term, so using a different $\sigma_i$ in each term does not matter as long as they satisfy the boundedness assumptions above.
    \item since we match the means of $\bg$ and those of $\bx$, the requirement that $\sigma$ be odd is unimportant in our setting.
\end{itemize}
In particular, for bounded Lipschitz test functions $\varphi$, the proof of Lemma 2 in \cite{hu2020universality} shows that for any $\btheta \in \dR^d$,
\begin{equation}\label{eq:app:hu_lu_bound}
\left| \Ea{\varphi(\btheta^\top \bx)} - \Ea{\varphi(\btheta^\top \bg)} \right| \leq \frac{C \norm{\btheta}_\infty \polylog(p)}{\nu^2}
\end{equation}
where $\nu^2$ is the variance of $\btheta^\top \bx$:
\begin{equation}
    \nu^2 = \btheta^\top \bSigma \btheta.
\end{equation}
We now place ourselves in the setting where $\btheta \in \cS_p$ where $\cS_p$ is defined in \eqref{eq:1d_clt_constraints}, and we consider two cases:
\begin{enumerate}
    \item if $\nu^2 > p^{-2\eta/3}$, then \eqref{eq:app:hu_lu_bound} reduces to
    \begin{equation}
        \left| \Ea{\varphi(\btheta^\top \bx)} - \Ea{\varphi(\btheta^\top \bg)} \right| \leq \frac{C \polylog(p)}{p^{\eta/3}}.
    \end{equation}
    \item if instead $\nu^2 > p^{-2\eta/3}$, then by the Lipschitz property of $\varphi$ we have
    \begin{equation}
        \left|\Ea{\varphi(\btheta^\top \bx)} - \varphi(\bmu)\right| \leq C \sqrt{\nu^2} = \frac{C}{p^{\eta/3}},
    \end{equation}
    and the same holds for $\bg$.
\end{enumerate}
In both cases, the bounds goes to $0$ uniformly over the whole constraint set $\cS_p$, which shows that Assumption \ref{assump:clt} holds.

We now move to checking Assumption \ref{assump:cov_bounded}; Lemma 8 in \cite{hu2020universality} (more precisely, eq. (159)) exactly shows that for $\btheta \in \cS_p$, the random variable $\btheta^\top x - \btheta^\top \bmu$ is $C$-subgaussian for an absolute constant $C$. Hence, we only need to show that $\bmu$ is uniformly bounded. Recall that
\[ \mu_i = \Ea{\sigma(\bff_i^\top \bz)} = \Ea{\sigma\left(\bff_i^\top \bmu^\bz + (1 + \tau_i) \tilde z\right)} \]
where $\tilde z$ is a standard normal variable and
\[ \tau_i = \norm{\bv_i} - 1. \]
Then, we can write using the Lipschitz property of $\sigma$
\[ \sigma(\bff_i^\top \bz) = \sigma(\tilde z) + (\bff_i^\top \bmu^\bz + \tau_i \tilde z) \tilde \sigma(\bff_i^\top \bz) \]
where $\tilde \sigma$ is a uniformly bounded function. By assumption, $\sigma(\tilde z)$ has zero mean, and hence by the Cauchy-Schwarz inequality
\begin{align*}
    \norm{\bmu}^2 &\leq C\left(\sum_{i\in [d]} (\bff_i^\top \bmu^\bz)^2 + (\norm{\bv_i} - 1)^2 \right) \\
    &\leq C \left(\norm{\bF}_{\mathrm{op}}^2 \norm{\bmu^\bz}^2 + \sum_{i \in [d]} \left|\norm{\bv_i}^2 - 1\right| \right) \\
    &\leq C'
\end{align*}
under the high-probability event $\cB$. \hfill \qedsymbol

\subsection{Proof of Theorem \ref{thm:weak_conv}}\label{app:weak_conv}

Our proof utilizes the  results in \cite{loureiro2021learning_gm} that describe the asymptotic limits of the estimators obtained through empirical risk minimization on the mixture of gaussians dataset. We note that the assumptions A1-A5 of their Theorem 1 are satisfied by our setting.

Let $\bW^{\star}$ denote the minimizer of the objective in equation \ref{eq:gmm_erm}, and let $\bZ^\star = \bX\bW^{\star}$. Let $\bxi_{k \in [K]}\sim \mathcal N(\mathbf 0,\bI_K)$ , $\bXi_k\in\mathbb R^{K\times d}$ be sets of $K$-dimensional vectors and  dimensional matrices respectively, with i.i.d entries sampled from $\mathcal N(0,1)$. 

Then, Theorem 1 in \cite{loureiro2021learning_gm} proves that for any pseudo-lipschitz functions of finite order, $\phi_{1}:\mathbb{R}^{K\times d} \to \mathbb{R}, \phi_{2}: \mathbb{R}^{K\times n}\to \mathbb{R}$:
\begin{align}\label{eq:pseudo_conv}
    \phi_{1}(\bW^{\star}) \xrightarrow[n,d \to +\infty]{P}\mathbb{E}_{\bXi}\left[\phi_{1}(\bG)\right], && \phi_{2}(\bZ^{\star})\xrightarrow[n,d \to +\infty]{P}\mathbb{E}_{\bxi}\left[\phi_{2}(\bH)\right]\,
\end{align}
Here $\bG$ and $\bH$ are functions of certain finite dimensional parameters
\[ \bu := (\bQ_{k} \in \R^{K \times K},\bM_{k} \in \R^K,\bV_{k} \in \R^{K \times K},\bhQ_{k} \in \R^{K \times K},\bhM_{k} \in \R^K,\bhV_{k} \in \R^{K \times K})_{k\in[K]} \] 
and the random vectors $\bxi_{k \in [K]},\bXi_{k \in [K]}$. The matrix $\bH$ is obtained by concatenating the following functions $\bh_k$, $\rho_{k}n$ time for each $k$:
\begin{equation} \bh_k=\bV_k^{1/2}\Prox_{\ell(\be_k,\bV_k^{1/2}\bullet)}(\bV^{-1/2}_k\bomega_{k})\in\mathbb R^{K}\, , \qquad
\boldsymbol{\omega}_k\equiv\bM_k+\bbb+\bQ^{1/2}_k\bxi_{k}\, ,
\end{equation}
Similarly, the matrix $\bG\in\mathbb R^{K\times d}$ is described by:
\begin{equation}
\bG=\bsA^{\frac{1}{2}}\odot \Prox_{r(\bsA^{\frac{1}{2}}\odot \bullet)}(\bsA^{\frac{1}{2}}\odot \bbB),\ \ \ \ 
\bsA^{-1}\equiv\sum_k \bhV_k\otimes\bSigma_k,\ \  \bbB\!\equiv\sum_k\!\left(\!\bmu_k\bhM_k^\top\!+\!\bXi_k\odot \sqrt{\bhQ_k\!\otimes\!\bSigma_k }\!\right). \nonumber
\end{equation}
Further define the function:
$\bff_k\equiv \bV_k^{-1}(\bh_k-\bomega_k)$. The equivalent bias vector $\bbb^\star$ is defined through the linear equation:
\begin{equation}
    \sum_k \rho_k\mathbb{E}_{\bxi}\left[\bV_k\bff_k\right]=\mathbf 0,
\end{equation}
and is therefore, unique, differentiable in $\bu$.
The parameters $\bu$ satisfy the following equations:
\begin{equation}\label{spgeneral}
\begin{cases}
\bQ_k\!=\!\frac{1}{d}\mathbb E_{\bXi}[\bG\bSigma_k\bG^\top]\\
\bM_k\!=\!\frac{1}{\sqrt d}\mathbb E_{\bXi}[\bG\bmu_k]\\
\bV_k\!=\!\frac{1}{d}\mathbb{E}_{\bXi}\!\!\left[\left(\bG\odot \!\left(\bhQ_k\otimes\bSigma_k\right)^{-\frac{1}{2}}\!\!\!\odot (\bI_K\otimes\bSigma_k) \!\right)\bXi_k^\top\right]
\end{cases} \!
\begin{cases}
\bhQ_k\!=\alpha \rho_k\mathbb E_{\bxi}\left[\bff_k\bff_k^\top\right]\\
\bhV_k\!=-\alpha \rho_k\bQ_k^{-\frac{1}{2}}\mathbb E_{\bxi}\!\left[\bff_k\bxi^\top\right]\\
\bhM_k\!=\alpha \rho_k\mathbb E_{\bxi}\left[\bff_k\right].
\end{cases}
\end{equation}.

We observe that the system of equations \eqref{spgeneral} can be expressed as a multi-dimensional fixed point equation:
\begin{equation}\label{eq:fixed_points}
    \bu = F_n(\bu).
\end{equation}

We make the following assumption, which is slightly stronger than (A5) in \cite{loureiro2021learning_gm}:
\begin{assump}\label{assump:bounded_unique}
 The fixed point equations $\bu = F_n(\bu)$ have unique solutions 
 $\forall n \in \mathbb{N}$.  Let $\hat{\bu}_n$ be the unique solution to $\bu=F_n(\bu)$. We further assume the solutions are uniformly bounded, i.e:
 \begin{equation}
     \norm{\hat{\bu}_n} \leq K
 \end{equation}
 for some constant $K$, and the jacobian of the fixed point equations $I-\frac{dF_n}{d\bu}$ is invertible. Furthermore, we assume that the same assumptions hold for the limiting equations $\bu=F(\bu)$.
\end{assump}

\textbf{Remark}: While we assume the above conditions, as noted in \cite{loureiro2021learning_gm}, the fixed point equations \ref{spgeneral}, correspond to the optimality conditions of a strictly convex-concave problem. This can be rigorously proven using the properties of Bregman envelopes, as in \cite{loureiro2021learning,celentano2020lasso}. The strict convexity-concavity then implies the uniqueness of the fixed points as well as the differentiability of the limits.

We now prove the following result:
\begin{lemma}\label{lem:fix_conv}
    Under assumption \ref{assump:emp_conv}, the system of equations \eqref{spgeneral} converge uniformly to a limiting system of equations $\bu=F(\bu)$.
\end{lemma}

\begin{proof}
We first show that the coordinates of the equivalent minimizer $\bG$ can be expressed as follows:
    \begin{equation}\label{eq:g_prox}
        \bG_i = g(\{\mu_{c, i}\}_{c \in \cC},\{\sigma_{c, i}\}_{c \in \cC},\xi_i,\bu)
    \end{equation}
Where $g$ is a differentiable function and $\xi_i$ denote independent Gaussian random variables. Indeed, from the separability assumption on $r$, and the definition of the prox operator, we have
\begin{align}
    \Prox_{r(\bsA^{\frac{1}{2}}\odot \bullet)}(\bsA^{\frac{1}{2}}\odot \bbB) &= \argmin_{\zz} r(\bsA^{1/2}\odot \zz)+ \frac12\norm{\zz-(\bsA^{1/2}\odot \bbB)}^2\\
    &=\argmin_{\zz} \sum_{i=1}^d \psi_r((\bsA^{1/2}\odot \bz)_i) + \frac12\sum_{i=1}^d (\bz_i-(\bsA^{1/2}\odot \bbB)_i)^2.
\end{align}
$\bu$ therefore only depends on the $i_{th}$ entry of $(\bsA^{1/2}\odot \bbB)$. Further, since all $\bSigma_c$ are assumed diagonal, the entries of $(\bsA^{1/2} \odot \bz)_i$ and $(\bsA^{1/2} \odot \bB)_i$ only depend on $\bz_i, \{\mu_{c, i}\}_{c \in \cC},\{\sigma_{c, i}\}_{c \in \cC},\xi_i$, and the parameters $\bu$ . The differentiability of $g$ then follows from the implicit function theorem applied to $\psi_r(\bsA^{1/2} \odot \bullet) + \sfrac12(\bullet - (\bsA^{1/2} \odot \bB)_i)^2$ . The same holds for the matrix $\bH$.

We next observe that each coordinate of $F_n$ of the above system of equations \ref{spgeneral} can be expressed as an expectation of a fixed continuous function w.r.t the joint empirical measure of $(\{\mu_{c, i}\}_{c \in \cC},\{\sigma_{c, i}\}_{c \in \cC})$.
    For instance, consider the $(i,j)_{th}$ entry of $\bQ_k$. We have:
\begin{equation}
    \bQ_{k, {ij}}=F_{q,i,j,n}=\frac{1}{d}\sum_{\ell=1}^d\mathbb E_{\bXi}[\bG_{i\ell}(\bSigma_k)_{\ell\ell}\bG_{\ell j}].
\end{equation}
Using equation \eqref{eq:g_prox}, we have that $G_{i\ell}$ only depends on the $\ell_{th}$ coordinates of the means, covariances. Therefore, for fixed $\hat{\bu}_n$,  $\bQ_{k, {ij}}$ is an expectation w.r.t the joint empirical measure of $(\{\mu_{c, i}\}_{c \in \cC},\{\sigma_{c, i}\}_{c \in \cC})$ of a continuous function. By Assumption \ref{assump:emp_conv}, we have that 
\[ \lim_{n \rightarrow \infty} F_{q,i,j,n} = F_{q,i,j},\] 
where $F_{q,i,j}$ denotes the expectation of $E_{\bXi}[\bG_{i\ell }(\bSigma_k)_{\ell \ell }\bG_{\ell j}]$ w.r.t the joint empirical measure.

To show that the convergence is uniform, we utilize assumptions \ref{assump:cov_bounded} and \ref{assump:bounded_unique}. Since each term in $F_n$ can be expressed as:
\begin{equation}
    F^{j}_n(\bu)=\frac{1}{d}\sum_{i=1}^d \Phi_j(\bu,\{\mu_{c, i}\}_{c \in \cC}, \{\sigma_{c, i}\}_{c \in \cC}),
\end{equation}
for some function $\Phi_j$ with Lipschitz constant $L_j$. Therefore, for any $n \in \mathbb{N}$ ,$F^{j}_n(\bu)$ is $L_j$ Lipschitz. This implies the uniform convergence of $F_n$ to $F$.
\end{proof}

We now use the above result to prove convergence of the sequence of solutions $\hat{\bu}_n$.
\begin{lemma}
 Under Assumptions \ref{assump:emp_conv} and \ref{assump:bounded_unique}:
\begin{equation}
   \lim_{n \rightarrow \infty} \hat{\bu}_n = \hat{\bu},
\end{equation}
where $\hat\bu$ is the solution to the limiting equations $\bu = F(\bu)$.
\end{lemma}
\begin{proof}
    By assumption, $\hat{\bu}_n$ are bounded. Therefore, by the Bolzano–Weierstrass theorem, there exists a convergent subsequence. Let $\hat{\bu}_{n_j}$ be any such subsequence with corresponding limit $\tilde{\bu}$. 
    We have:
    \begin{align*}
        \norm{F(\tilde{\bu})-\tilde{\bu}} &\leq \norm{F(\tilde{\bu})-F(\hat{\bu}_{n_j})}+\norm{F(\hat{\bu}_{n_j})-F_{n_j}(\hat{\bu}_{n_j})}.
    \end{align*}
By uniform convergence of $F_n$ to $F$ (Lemma \ref{lem:fix_conv}, we have that $\norm{F(\hat{\bu}_{n_j})-F_{n_j}(\hat{\bu}_{n_j})} \rightarrow 0$ while $\norm{F(\tilde{\bu})-F(\hat{\bu}_{n_j})} \rightarrow 0$ from the convergence of $\hat{\bu}_{n_j}$ to $\tilde{\bu}$ and the continuity of $F$.
Therefore, we must have $F(\tilde{\bu})-\tilde{\bu} = 0$ and thus $\tilde{\bu} = \hat{\bu}$. Therefore, any convergence subsequence of $\hat{\bu}_n$ converges to $\hat{\bu}$. Since $ \hat{\bu}_n$ are bounded, this implies that  $\lim_{n \rightarrow \infty} \hat{\bu}_n =\hat{\bu}$.
\end{proof}

Now, let $\Phi$ be a twice differentiable test function  with a Hessian having bounded spectral norm. We define:
\begin{equation}\label{eq:h_sum}
    h_d(\bW)=\frac{1}{d}\sum_{i=1}^d \Phi(\{W_{c, i}\}_{c \in \cC},\{\mu_{c, i}\}_{c \in \cC},\{\sigma_{c, i}\}_{c \in \cC}).
\end{equation}
Let $\hat{\bu}(s)$ denote the solution to the limiting equation $\bu=F_s(\bu)$ defined in Lemma \ref{lem:fix_conv} for regularization
\[ r_s(\bW) = \frac\lambda n r_d(\bW) + s h_d (\bW) \]
By Assumption \ref{assump:bounded_unique}, $\bI-\frac{dF_s}{d\bu}$ is invertible in a neighbourhood of $0$. Therefore, by implicit function theorem and uniqueness of the fixed points, $\hat{\bu}(s)$, is a continuously differentiable function of $s$.

Using equation \eqref{eq:g_prox}, we obtain that $G$ is a differentiable function of $s$. 
We now consider the perturbed training loss:
\begin{equation}
\hR(\bW,\bbb,s)=\frac{1}{n}\sum_{i=1}^n\ell\left(\frac{\bW\bx_i}{\sqrt d}+\bbb, \by_i\right) + r_s(\bW).
\end{equation}
By the boundedness of the Hessian of $\Phi$, $\hR(\bW,\bbb,s)$ satisfies the assumptions of strict convexity, coercivity for small enough $s$.
We first note from \eqref{eq:pseudo_conv} and Theorem 2 in \cite{loureiro2021learning_gm}, the expected training error converges to the following limit:
\begin{equation}\label{eq:loss_conv}
\Eb{\bX}{\frac{1}{n}\sum_{i=1}^n\ell\left(\frac{\bW\bx_i}{\sqrt d}+\bbb, \by_i\right)} \xrightarrow[n,d \to +\infty]{} \sum_{k=1}^K\rho_k\mathbb E_{\bxi}[\ell(\be_k,\bh_k)].
\end{equation}
Similar to equation \eqref{eq:g_prox}, we have that $\bh_k$ are continuous, differentiable functions of $\hat{\bu}(s)$.

From Assumption \ref{assump:reg} and the form of the perturbation \ref{eq:h_sum}, we further have that $\lambda r(\bW) + s h(\bW)$ is a pseudo-Lipschitz function of $\bW$ of finite order.
Therefore, Equation \eqref{eq:pseudo_conv} gives:
\begin{equation}
    \Eb{\bX}{\frac{1}{d}\lambda r(\bW^*) + s h(\bW^*)}  \xrightarrow[n,d \to +\infty]{} \Eb{\bXi}{\frac{1}{d}\lambda r_d(\bG_n) + s h_d(\bG_n)}.
\end{equation}
Define:
\begin{align}
   q_n(s,\bu)&= \frac{1}{d}\sum_{i=1}^d\Eb{\bXi}{\psi_r(g(\{\mu_{c, i}\}_{c \in \cC},\{\sigma_{c, i}\}_{c \in \cC}, \bSigma^k_i,\xi_i,\bu_n))}\\&+s\frac{1}{d}\sum_{i=1}^d\Eb{\bXi}{\Phi(g(\{\mu_{c, i}\}_{c \in \cC},\{\sigma_{c, i}\}_{c \in \cC}, \bSigma^k_i,\xi_i,\bu),\{\mu_{c, i}\}_{c \in \cC},\{\sigma_{c, i}\}_{c \in \cC})}.
\end{align}
From assumption \ref{assump:reg} and equation \ref{eq:g_prox}, we have:
\begin{equation}
    \Eb{\bXi}{\frac{1}{d}\lambda r_d(\bG_n) + s h_d(\bG_n)}=q_n(s, \hat{\bu}_n(s)).
\end{equation}
Similar to the proof of Lemma \ref{lem:fix_conv}, we obtain that $q_n(s,\bu)$ converges uniformly to $q(s,\bu)$ given by the corresponding expectation w.r.t the limiting empirical measure:
\begin{align}
   q(s,\bu)= \Eb{p(\bm\sigma, \bm\mu)}{\Eb{\bXi}{\Phi(g(\{\mu_{c, i}\}_{c \in \cC},\{\sigma_{c, i}\}_{c \in \cC}, \bSigma^k_i,\xi_i,\bu),\{\mu_{c, i}\}_{c \in \cC},\{\sigma_{c, i}\}_{c \in \cC})}}.
\end{align}
Due to the uniform convergence, we further have:
\begin{equation}
    \lim_{n \rightarrow \infty} q_n(s,\hat{\bu_n}(s)) = q(s,\hat{\bu}(s)).
\end{equation}
We conclude that:
\begin{equation}\label{eq:diff_lim_fin}
    \Eb{\bX}{\frac{1}{n}\sum_{i=1}^n\ell\left(\frac{\bW^*\bx_i}{\sqrt d}+\bbb^*, \by_i\right)+\frac{1}{d}\lambda r(\bW^*) + s h(\bW^*)} \rightarrow \sum_{k=1}^K\rho_k\mathbb E_{\bxi}[\ell(\be_k,\bh_k)]+q(s,\hat{\bu}(s)).
\end{equation}

Since the RHS is a differentiable function in $s$, Assumption \ref{assump:diff_lim} is satisfied for the perturbation $h(\bW)$. 
Due to the coercivity of $\ell(\by,\bullet\bX)+r(\bullet)$, there exists a sequence of fixed compact subsets containing the minimizers $\bW^*$ with high probability as $n \rightarrow \infty$ (see Lemma 8 in \cite{loureiro2021learning}).
Furthermore, since the input distribution is given by a mixture of gaussians with bounded means, Assumption \ref{assump:emp_conv} is satisfied for any such sequence of constraint sets.
Therefore, the validity of assumption \ref{assump:diff_lim} through Equation \ref{assump:diff_lim} allows the applicability of 
Theorem \ref{thm:overlap_universality} for the statistic $h_d(\bW)$. Through standard approximation technniques or the Stone–Weierstrass theorem, the restriction of differentiability and bounded Hessian of $\Phi$ can be removed. This completes the proof of Theorem \ref{thm:weak_conv} for general bounded Lipschitz $\Phi$.

\section{Assumptions on the target function}
\label{sec:app:target}
In this section, we discuss possible generalizations of the assumptions on the target function. In \eqref{eq:def:teacher}, we assume a target function depending on a small number of linear projections in the input space, along with the class labels. However, when the inputs are generated through feature maps $\bm x =\psi(\bz)$, one may instead consider target functions depending directly on the latent vectors $\bz$. This was the setup considered in \cite{hu2020universality} for random feature maps. For mixture models considered in our work, one may assume:

\begin{align}
\label{eq:def:teacher_rf}
y_{i}(\bX) = \eta(\bm \Theta_{\star}^{\top} \bm z_i, \eps_i, c_i).
\end{align}

We conjecture that our results can be generalized to the above setup through the use of the following stronger assumption:

\begin{assump}
\label{assump:joint_clt}
For any Lipschitz function $\varphi: \dR \to \dR$, 
\begin{equation}\label{eq:app:feature_clt}
    \lim_{n, p \to \infty} \sup_{\bTheta_1\in\cS^{\bm x}_p,\bTheta_2\in\cS^{\bm z}_d} 
    \left| 
    \Ea{\varphi(\bTheta_1^\top \bm x,\bTheta_2^\top \bm z)\,\big|\,c_{\bm x}=c} - 
   \Ea{\varphi(\bTheta_1^\top\bm g,\bTheta_2^\top \bm z)\,\big|\,c_{\bm g}=c}
   \right| = 0,\quad \forall c \in \cC.
\end{equation}
Here $S^{\bm x}_p$ is the constraint set for the training parameters $\bTheta_1$ while $\cS^{\bm z}_d$ denotes a suitable constraint set on $\R^d$ where $d$ denotes the dimension of the latent vectors. Under the above assumption
\end{assump}

Such an assumption has been discussed in \cite{goldt2019modelling} under the term ``Hidden Manifold Model", and was proven in \cite{hu2020universality} for random feature maps acting on Gaussian noise. 

\section{One-dimensional gaussian approximation}
\label{app:1d_clt_arguments}

Although Theorem \ref{thm:montanari_universality} is a powerful result, it still relies on very strong assumptions. In particular, given a distribution $\nu$ for the inputs $\bm x_i$, characterizing the set of vectors $\bm \theta$ such that Assumption \ref{assump:clt} holds is in general a difficult task.

\paragraph{Rigorous results} When the entries of $\bm x$ are i.i.d subgaussian, a classical application of the Lindeberg method \citep{lindeberg_1922_eine} shows that Assumptions \ref{assump:concentration} and \ref{assump:clt} are satisfied with
\[ \cS_p = \Set*{\bm \theta \in \dR^p \given \norm{\bm \theta}_\infty  = o_p(1)}. \]
More recently, this result (often used under the name ``Gaussian Equivalence Theorem'') was extended to general feature models with approximate orthogonality constraints \citep{hu2020universality, goldt2020gaussian}, for the same choice of $\cS_p$. \cite{montanari2022universality} also provides a central limit theorem result for the Neural Tangent Kernel of \citep{jacot2018neural}, for a more convoluted parameter set $\cS_p$. While these papers provide a strong basis for the one-dimensional CLT, those rigorous results only concern (so far) a very restricted set of distributions.

\paragraph{Concentration of the norm} Another, more informal line of work originating from \cite{seddik_2020_random}, argues that most distributions found in the real world satisfy some form of the central limit theorem. The starting point of this analysis is the following theorem, adapted from \cite{bobkov_concentration_2003}:
\begin{theorem}[Corollary 2.5 from \cite{bobkov_concentration_2003}]\label{thm:almost_everywhere_clt}
Let $\bm x \in\dR^p$ be a random variable, with $\E*{\bm x\bm x^\top} = \bm I_p$, and $\eta_p$ the smallest positive number such that
\begin{equation}\label{eq:norm_concentration}
\Pb*{\left| \frac{\norm{\bm x}_2}{\sqrt{p}} - 1 \right| \geq \eta_p} \leq \eta_p.
\end{equation}
Then for any $\delta > 0$, there exists a subset $\cS_p$ of the $p$-sphere $\dS^{p-1}$ of measure at least $4p^{3/8}e^{-c p\delta^4}$, such that
\[ \sup_{\bm \theta\in \cS_p} \sup_{t\in \dR} \left|\Pb{\bm \theta^\top \bm x \geq t} - \Phi(t) \right| \leq \delta + 4\eta_p,\]
where $\Phi$ is the characteristic function of a standard Gaussian, and $c$ is a universal constant.
\end{theorem}
If both $\delta$ and $\eta_p$ are $o(1)$, Theorem \ref{thm:almost_everywhere_clt} implies that Assumption \ref{assump:clt} is satisfied for any compact subset $\cS'_p \subseteq \cS_p$. This suggests that the norm concentration property of \eqref{eq:norm_concentration} is a convenient proxy for one-dimensional CLTs. However, the proof of this theorem uses isoperimetric inequalities, and is thus non-constructive; as a result, characterizing precisely the set $\cS_p$ remains an open and challenging mathematical problem.

\paragraph{Concentrated vectors} In \cite{seddik_2020_random}, the authors consider the concept of \emph{concentrated} random variables, as defined in \cite{ledoux_2001_concentration}:
\begin{definition}\label{def:concentrated}
Let $\bm x\in \dR^p$ be a random vector. $\bm x$ is called (exponentially) concentrated if there exists two constants $C, c$ such that for any 1-Lipschitz function $f: \dR^p \to \dR$, we have
\[ \Pb{\left|f(\bm x) - \E{f(\bm x)} \right|   \geq t} \leq C e^{-ct^2}.\]
\end{definition}
Since the norm function is $1$-Lipschitz, it can be shown that any concentrated isotropic vector $\bm x$ satisfies \eqref{eq:norm_concentration}, with
\[ \eta_p \propto \left(\frac{\log(p)}{p}\right)^{1/2} \]
The converse is obviously not true; an exponential random vector still has $\eta_p \to 0$, but is not concentrated. However, even if it is stronger that \eqref{eq:norm_concentration}, the concept of concentrated vectors has two important properties:
\begin{enumerate}
    \item a standard Gaussian vector $\bm x \sim \cN(\bm 0, \bm I_p)$ satisfies Definition \ref{def:concentrated} with constants $C, c$ independent from $p$,
    \item if $\bm x \in \dR^p$ is a concentrated vector with constants $C, c$ and $\Psi: \dR^p \to \dR^{q}$ is an $L$-Lipschitz function, then $\Psi(\bm x)$ is also a concentrated vector, with constants only depending on $c, C$ and $L$.
\end{enumerate}

\paragraph{Towards real-world datasets} The real-world data considered in machine learning is often composed of very high-dimensional inputs, corresponding to $p \gg 1$ in our setting. However, it is generally accepted that this data actually lies on a low-dimensional manifold of dimension $d_0$: this is the idea behind many dimensionality reduction techniques, from PCA \citep{pearson_1901_lines} to autoencoders \citep{kramer_1991_nonlinear}. Another, more recent line of work (see e.g. \cite{facco2017estimating}) studies the estimation of the latent dimension $d_0$; results for the MNIST dataset ($p=784$) yield $d_0 \approx 15$, while CIFAR-10 ($p=3072$) has estimated intrinsic dimension $d_0\approx 35$ \citep{spigler2019asymptotic}.

Following this heuristic, the most widely used method to model realistic data is to learn a map $f: \dR^{d_0} \to \dR^p$, usually through a deep neural network, and then generate the $x_i$ according to
\begin{equation}
    \bm x = f(\zz) \quad \text{with} \quad \zz\sim \cN(\bm 0, \bm I_{d_0})
\end{equation} 
Examples of functions $f$ include GANs \citep{goodfellow2014generative}, variational auto-encoders \citep{kingma2013auto}, or normalizing flows \citep{rezende_2015_variational}. This ansatz has been studied theoretically, and the results compared with real-world datasets, in \cite{goldt2019modelling, loureiro2021learning}; the results indicate significant agreement between generated inputs and actual data.

Finally, we argue that for a large class of generative networks, the learned function $f$ is actually Lipschitz, with a bounded constant. This is even often a design choice; indeed, theoretical results such as \cite{bartlett_2017_spectrally} imply that a smaller Lipschitz constant improve the generalization capabilities of a network, or its numerical stability \citep{behrmann_2021_understanding}. As a result, regularizations aimed at controlling the Lipschitz properties of a network are a common occurrence; see e.g. \cite{miyato_2018_spectral} for the spectral regularization of GANs. This indicates that concentrated vectors are indeed a good approximation for real-world data. 
\section{Details on the numerical simulations}
\label{sec:app:numerics}
In this appendix we expand on how Fig.~\ref{fig:experiments} was generated. This closely follows the pipeline illustrated in Fig.~\ref{fig:cartoon_theorem1}.
\paragraph{Step 1: Training of the cGAN ---} The first step consists on training a cGAN on the real data set. For Fig.~\ref{fig:experiments}, we have used a PyTorch \citep{NEURIPS2019_9015} implementation of the architecture in \cite{Chen2016} publicly available at the \href{https://github.com/znxlwm/pytorch-generative-model-collections}{pytorch-generative-model-collections} repository. The cGAN was trained on the fashion-MNIST dataset \cite{xiao2017} following the default procedure in the repository: i.e. training for 50 epochs and batch size $64$ using Adam with learning rate $0.0002$ and $(\beta_{1},\beta_{2})=(0.5, 0.999)$ on the binary cross entropy (BCE) loss (equal for both generator and discriminator). The evolution of the training loss during training is given in Fig.~\ref{fig:app:loss}, and samples from the generator during different epochs are shown in Fig.~\ref{fig:app:samples}. 
\begin{figure}[t]
    \centering
    \includegraphics[width=0.6\textwidth]{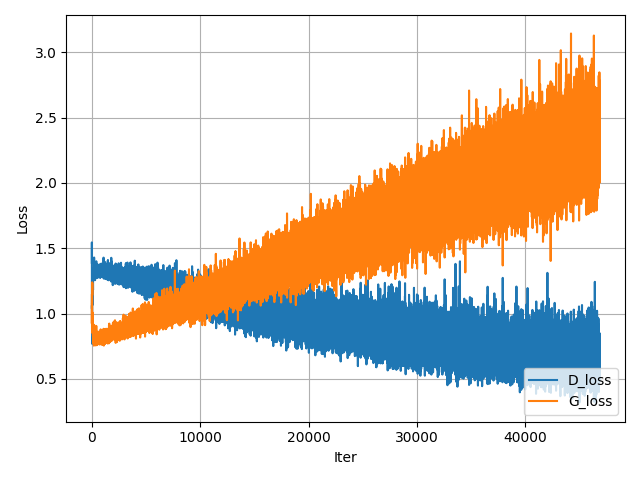}
    \caption{Evolution of the binary cross entropy training loss for the generator (orange) and discriminator (blue) during training.}
    \label{fig:app:loss}
\end{figure}

\begin{figure}[t]
    \centering
    \includegraphics[width=0.3\textwidth]{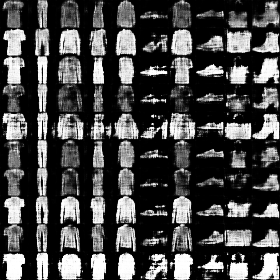}
    \includegraphics[width=0.3\textwidth]{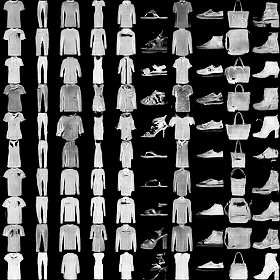}
    \includegraphics[width=0.3\textwidth]{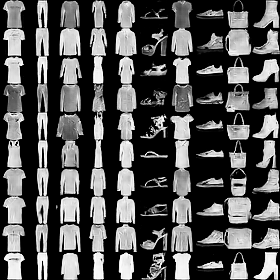}
    \caption{Fake fashion-MNIST Samples from the cGAN generator at the end of epoch 1 (left), 25 (middle) and 50 (right).}
    \label{fig:app:samples}
\end{figure}
\paragraph{Step 2: Evaluating the class means and covariances ---} With the trained cGAN in hands, we can generate as many fresh samples as needed for our experiments. Moreover, a feature map can be easily added on the top of the cGAN architecture, as illustrated in Fig. \ref{fig:cartoon_theorem1}. For Fig.~\ref{fig:experiments}, we have added a random feature map \citep{rahimi2007random} $\bx\mapsto \tanh(\mat{F}\bx)$ with projection matrix $\mat{F}\in\mathbb{R}^{1176\times 784}$ with entries $\mat{F}_{ij}\sim\mathcal{N}(0,\sfrac{1}{d})$. In order to compare the performance of a model trained on cGAN+RF samples vs. the equivalent Gaussian mixture model, we need to compute the class-wise means and covariances $(\bmu_{c}, \bSigma_{c})$. For Fig.~\ref{fig:experiments}, this was done with a standard Monte Carlo scheme over $10^{6}$ samples. 
\paragraph{Step 3: Learning curves ---} The last step consists of computing the curves for the test error of logistic and ridge regression trained on the cGAN+RF features. Each point in Fig.~\ref{fig:experiments} corresponds to a fixed sample complexity $\alpha = \sfrac{n}{p}$. For each $\alpha$, we generate fresh $n=\alpha\times p$ training features either from the cGAN+RF model (blue points) or from the equivalent Gaussian mixture model (red points). For the binary classification task, we split the samples over even vs. odds class labels. The SciPy \citep{2020SciPy-NMeth} implementation of both ridge and logistic regression were used to train a classifier on the training data, from which both training error and test error were computed, using another batch of fresh samples for the latter. Finally, to reduce finite-size effects this procedure was repeated over 10 different seeds, with the average and standard deviation reported in Fig.~\ref{fig:experiments}.

\end{document}